\documentclass[11pt]{amsart}
\usepackage{amsmath}
\usepackage{amssymb}
\usepackage{amsfonts}
\usepackage{amsthm}
\usepackage{enumerate}
\usepackage[mathscr]{eucal}
\usepackage{verbatim}
\usepackage{amsthm}
\usepackage{amscd}

\newtheorem{theorem}{Theorem}[section]

\newtheorem{lemma}[theorem]{Lemma}

\theoremstyle{definition}
\newtheorem{definition}{Definition}
\newtheorem{remark}{Remark}

\numberwithin{equation}{section}

\newcommand\relphantom[1]{\mathrel{\phantom{#1}}}

\begin{document}

\address{Department of Mathematics \\
           University of Wisconsin - Madison\\
            , U.S.A.}
   \email{park@math.wisc.edu}

\author{Bae Jun Park}

\title[Function spaces]{On the boundedness of pseudo-differential operators on Triebel-Lizorkin and Besov spaces}

\begin{abstract} 
In this work we show endpoint boundedness properties of pseudo-differential operators of type $(\rho,\rho)$, $0<\rho<1$, on Triebel-Lizorkin and Besov spaces.  Our results are sharp and they also cover operators defined by compound symbols.

\end{abstract}

\maketitle

\section{\textbf{Introduction and main results}}\label{intro}

Let $S(\mathbb{R}^d)$ denote the Schwartz space and $S'(\mathbb{R}^d)$ the space of tempered distributions. 
For $f\in S(\mathbb{R}^d)$ the Fourier transform is defined by the formula
\begin{equation*}
\widehat{f}(\xi):=\int_{\mathbb{R}^d}{f(x) e^{-2\pi i\langle x,\xi\rangle}}dx ~~~(\xi \in\mathbb{R}^d)
\end{equation*} and denote by $f^{\vee}$ the inverse Fourier transform of $f$.  We also extend these transforms  to the space of tempered distributions.

A symbol $a$ in H\"ormander's class $\mathcal{S}^m_{\rho,\delta}$ is a smooth function defined on $\mathbb{R}^d\times\mathbb{R}^d$, satisfying that for all multi-indices $\alpha$ and $\beta$ there exists a constant $C_{\alpha,\beta}$ such that 
\begin{equation}\label{symbolest}|\partial_{\xi}^{\alpha}\partial_{x}^{\beta}a(x,\xi)|\leq C_{\alpha,\beta}(1+|\xi|)^{m-\rho|\alpha|+\delta|\beta|} ~ \text{for}~ (x,\xi)\in\mathbb{R}^d\times \mathbb{R}^d,
\end{equation} 
and the corresponding pseudo-differential operator $T_a$ is given by 
\begin{equation}\label{symbolcal}
T_af(x):=\int_{\mathbb{R}^d}{a(x,\xi)\widehat{f}(\xi)e^{2\pi i \langle x,\xi \rangle}}d\xi, \quad f\in\mathcal{S}(\mathbb{R}^d).
\end{equation}
Denote by  $Op\mathcal{S}_{\rho,\delta}^m$ the class of pseudo-differential operators with symbols in $\mathcal{S}_{\rho,\delta}^{m}$.
In \cite{Ho}, \cite[p94]{Ho1} H\"ormander  showed that  for  $0\leq \delta < \rho <1$ the adjoint operator of $T_a\in Op\mathcal{S}_{\rho,\delta}^{m}$ belongs to the same type of class by using an asymptotic expansion, and  in this case $(T_a)^*=T_{a^*}$ where
\begin{equation*}
a^*(x,\xi)=\int_{\mathbb{R}^d\times\mathbb{R}^d}{{\overline{a}(x-y,\xi-\eta)e^{-2\pi i\langle y,\eta\rangle }}d\eta}dy,
\end{equation*} interpreted suitably as an oscillatory integral.  He also mentioned that this is also true when  $0\leq \delta=\rho <1$  and we will give a proof in Appendix.
The operator $T_a$ is well-defined on $S(\mathbb{R}^d)$ and it maps $S(\mathbb{R}^d)$ continuously into itself. This extends via duality to a mapping from $S'(\mathbb{R}^d)$ to itself.

We now recall the definitions of Besov spaces  and Triebel-Lizorkin spaces from \cite{Fr_Ja, Tr}. Let $0<p,q\leq \infty$ and $s\in\mathbb{R}$.
 Let $\Phi$ be a system of functions $\{\omega_k\}$ in $S(\mathbb{R}^d)$   satisfying
\begin{equation*}
\sum_{k=0}^{\infty}{\omega_k(x)}=1 \quad \text{for}\quad x\in\mathbb{R}^d
\end{equation*}
\begin{equation*}
Supp{(\omega_0)}\subset \{x:|x|\leq 2\}
\end{equation*}
\begin{equation*}
Supp{(\omega_k)}\subset\{x:2^{k-1}\leq |x|\leq 2^{k+1}\} \quad\text{if}\quad k=1,2,3,\dots
\end{equation*}
\begin{equation*}
\sup_{{x\in\mathbb{R}^d, k\in\mathbb{Z}}}{2^{k|\alpha|}|\partial^{\alpha}\omega_k(x)|}<\infty \quad\text{for every multi-index}\quad \alpha.
\end{equation*}
For a fixed choice of such a system $\Phi$, define Besov spaces $B_p^{s,q}$ and  Triebel-Lizorkin spaces ${F}_p^{s,q}$ as
\begin{equation*}
{B}_p^{s,q}:=\{f\in S'(\mathbb{R}^d): \big\Vert  f\big\Vert_{{F}_p^{s,q}}:=\big\Vert  \{2^{ks} (\omega_k\widehat{f})^{\vee} \}  \big\Vert_{l^q(L^p)}<\infty \},   
\end{equation*}
\begin{equation}\label{triebelnorm}
{F}_p^{s,q}:=\{f\in S'(\mathbb{R}^d): \big\Vert  f\big\Vert_{{F}_p^{s,q}}:=\big\Vert  \{2^{ks} (\omega_k\widehat{f})^{\vee} \}  \big\Vert_{L^p(l^q)}<\infty \},   \quad p<\infty.
\end{equation}

As shown in \cite[2.3.2]{Tr} the spaces do not depend on the choice of $\Phi$.
On the other hand, an extension of (\ref{triebelnorm}) to $p=\infty$ does not make sense (unless $q=\infty$, in which case $F_{\infty}^{s,\infty}=B_{\infty}^{s,\infty}$) because of the dependence of $\Phi$. For details see  \cite[2.1.4]{Tr}.
 Alternatively, we define 
  \begin{equation*}
 \Vert f\Vert_{F_{\infty}^{s,q}}:=\big\Vert \big(  \omega_0 \widehat{f} \big)^{\vee} \big\Vert_{L^{\infty}}+\sup_{P:l(P)<1}{\Big(\frac{1}{|P|}\int_P{\sum_{k=-\log_2{l(P)}}^{\infty}\big| \big( 2^{ks}\omega_k\widehat{f}  \big)^{\vee}(x)\big|^q}dx\Big)^{1/q}}
 \end{equation*} where the supremum is taken over all dyadic cubes $P$ of sidelength $l(P)<1$, and
\begin{equation*}
F_{\infty}^{s,q}:=\Big\{f\in S' : \Vert f\Vert_{F_{\infty}^{s,q}}<\infty\Big\}
\end{equation*} 
which are analogous to characterizations of $bmo$ via Carleson measures.
Then this definition is independent of the choice of $\omega_k\in \Phi$ (See \cite{Bu_Ta, Fr_Ja1}).

In this section we fix such a system and use a notation $\phi_k(x)=\omega_k^{\vee}(x)$ where $\{\phi_k\}_{k=0,1,2,\dots}$ is a Littlewood-Paley partition of unity. That is, $\phi_0$ and $\phi$ are Schwartz functions satisfying $Supp(\widehat{\phi_0})\subset \{\xi: |\xi|\leq 2\}$, $Supp(\widehat{\phi})\subset \{\xi: 2^{-1}\leq |\xi|\leq 2\}$, and $\sum_{k=0}^{\infty}{\widehat{\phi_k}(\xi)}=1$ for $\xi\in\mathbb{R}^d$ where $\phi_k(x):=2^{kd}\phi(2^kx)$ for $k\geq 1$.

Then the (quasi-)norms on the spaces are 
 \begin{equation*}
 \Vert f\Vert_{{B}_p^{s,q}}= \Big(\sum_{k=0}^{\infty}{(2^{s k}\big\Vert \phi_k\ast f\big\Vert_{L^p})^q}\Big)^{{1}/{q}} , 
 \end{equation*}
 \begin{equation*}
 \Vert f\Vert_{{F}_p^{s,q}}=\Big\Vert  \Big(\sum_{k=0}^{\infty}{(2^{s k}|\phi_k\ast f|)^q}\Big)^{{1}/{q}}  \Big\Vert_{L^p}, \quad p<\infty
 \end{equation*}
 and
 \begin{equation*}
 \Vert f\Vert_{F_{\infty}^{s,q}}=\big\Vert \phi_0\ast f \big\Vert_{L^{\infty}}+\sup_{l(P)<1}{\Big( \frac{1}{|P|}\int_P{\sum_{k=-\log_2{l(P)}}^{\infty}{2^{skq}|\phi_k\ast f(x)|^q}}dx   \Big)^{1/q}}.
 \end{equation*}
According to those norms,  the spaces are quasi-Banach spaces (Banach spaces if $p\geq 1, q\geq 1$).

Note that those are a generalization of many standard function spaces such as $L^p$ spaces, Sobolev spaces, and Hardy spaces. We recall 
\begin{align*}
L^p = F_p^{0,2}, &\qquad 1<p<\infty\\
h^p = F_p^{0,2}, &\qquad 0<p<\infty\\
 {L}_s^p = {F}_p^{s,2}, &\qquad s>0, 1<p<\infty \\
 bmo = {F}_{\infty}^{0,2}. &
\end{align*}
where $h^p$ denotes the local Hardy spaces, introduced by Goldberg \cite{Go}.

The multiplier operator 
\begin{equation}\label{multiplierexample}
c_{m,\rho}(D):=\dfrac{e^{ -2\pi i |D|^{{(1-\rho)}}}}{(1+|D|^2)^{-{m}/{2}}}
\end{equation} is a typical example of translation invariant $Op\mathcal{S}_{\rho,\rho}^{m}$-operators and there are several boundedness results. Fefferman \cite{Fe0}, Hirschman \cite{Hi}, Stein \cite{St}, and Wainger \cite{Wa}   proved that for $1<p<\infty$ and $0<\rho<1$, $c_{m,\rho}(D)$ extends to a bounded operator on $L^p$ if and only if \begin{equation}\label{cexample}
m\leq-d(1-\rho)\big| 1/2-1/p \big|.
\end{equation} Pramanik, Rogers, and Seeger \cite{Pr_Ro_Se} showed  that when $2<p<\infty$, $c_{m,\rho}(D)$ maps $F_p^{0,p}$ into $F_p^{0,q}$ for any $0<q\leq \infty$.

Some boundedness results also hold for general $Op\mathcal{S}_{\rho,\rho}^{m}$-operators  when $0 <\rho <1$. Calder\'on and Vaillancourt \cite{Ca_Va} proved the $L^2$ boundedness if $m=0$ by using an almost orthogonality technique in a Hilbert space. Fefferman \cite{Fe} generalized this result to $L^p$ boundedness when $1<p<\infty$ with the condition (\ref{cexample}) by using an interpolation theorem in \cite{Fe_St}. P\"aiv\"arinta and Somersalo \cite{Pa_So} proved that these operators map $h^p$ into itself for $0<p<\infty$ if (\ref{cexample}) holds.

In this paper we extend the boundedness result of $c_{m,\rho}(D)$ in \cite{Pr_Ro_Se} to general $Op\mathcal{S}_{\rho,\rho}^{m}$-operators for all $0<p,q\leq \infty$. 

\begin{theorem}\label{main}
Let $0 <  \rho <1$, $0<p<\infty$, $0<q,t\leq\infty$, and $s_1,s_2 \in \mathbb{R}$. Suppoe $m\in\mathbb{R}$ satisfies \begin{equation}\label{ccexample}
m-s_1+s_2\leq -d(1-\rho)\big|  1/2-1/p \big|
\end{equation} and $a\in\mathcal{S}_{\rho,\rho}^m$. Then $T_a$ maps $F_p^{s_1,q}$ into $F_p^{s_2,t}$ if one of the following conditions holds.
\begin{enumerate}
\item   $m-s_1+s_2<-d(1-\rho)\big| 1/2-1/p  \big|$,
\item  $p=2$, $q\leq 2\leq t$, and $m=0$,
\item $0<p<2$, $p\leq t\leq\infty$, $0<q\leq \infty$, and $m-s_1+s_2=-d(1-\rho)\big(1/p-1/2\big)$,
\item $2< p< \infty$, $0<t\leq\infty$, $0<q\leq p$,  and $m-s_1+s_2=-d(1-\rho)\big(1/2-1/p\big)$.

\end{enumerate}

\end{theorem}

\begin{theorem}\label{besov}
Let $0 < \rho<1$, $0<p<\infty$, $0<q,t\leq \infty$, and $s_1,s_2\in \mathbb{R}$. Suppose $m\in\mathbb{R}$ satisfies $(\ref{ccexample})$ and $a\in \mathcal{S}_{\rho,\rho}^{m}$. Then $T_a$ maps $B_p^{s_1,q}$ into $B_p^{s_2,t}$ if one of the following conditions holds.
\begin{enumerate}
\item   $m-s_1+s_2<-d(1-\rho)\big| 1/2-1/p  \big|$,
\item  $q\leq t$ and $m-s_1+s_2=-d(1-\rho)\big|1/2-1/p\big|$.

\end{enumerate}
\end{theorem}

When $p=\infty$ the same boundedness results hold, but due to different definition and ideas of proof, we state the results separately. 

\begin{theorem}\label{mmm}
Let $0<\rho<1$, $0<q,t\leq\infty$ and $s_1,s_2\in\mathbb{R}$. Suppose $m\in\mathbb{R}$ satisfies $(\ref{ccexample})$ and $a\in\mathcal{S}_{\rho,\rho}^{m}$. Then $T_a$ maps $F_{\infty}^{s_1,q}$ into $F_{\infty}^{s_2,t}$.
\end{theorem}

\begin{theorem}\label{mmmm}
Let $0<\rho<1$, $0<q,t\leq \infty$, and $s_1,s_2\in\mathbb{R}$. Suppose $m\in\mathbb{R}$ satisfies $(\ref{ccexample})$ and $a\in\mathcal{S}_{\rho,\rho}^{m}$. Then $T_a$ maps $B_{\infty}^{s_1,q}$ into $B_{\infty}^{s_2,t}$ if one of the following conditions holds.
\begin{enumerate}
\item  $m-s_1+s_2<-d(1-\rho)/2$,
\item  $q\leq t$ and $m-s_1+s_2=-d(1-\rho)/2$.
\end{enumerate}
\end{theorem}

\begin{remark}
All of our results are sharp in the sense that the hypothesis (\ref{ccexample}) is necessary and when the equality of (\ref{ccexample}) holds the restrictions on $q,t$ are necessary. 
To be specific, we will prove that  the boundedness results fail with $c_{m,\rho}(D)\in Op\mathcal{S}_{\rho,0}^{m}$, defined in (\ref{multiplierexample}), provided that the assumptions do not work.
\end{remark}
\begin{theorem}\label{sharptheorem1}
Let $0<\rho<1$, $0<p,q,t\leq \infty$,  $s_1,s_2\in\mathbb{R}$, and $m\in\mathbb{R}$. Define $c_{m,\rho}(D)\in \mathcal{S}_{\rho,0}^{m}$ as in $(\ref{multiplierexample})$.
Then 
\begin{equation*}
\Vert c_{m,\rho}(D)\Vert_{F_p^{s_1,q}\to F_p^{s_2,t}}=\infty
\end{equation*} if one of the followings holds.
\begin{enumerate}
\item $m-s_1+s_2>-d(1-\rho)\big|1/2-1/p \big|$,
\item  $m-s_1+s_2=-d(1-\rho)\big(1/p-1/2 \big)$, $0<p\leq 2$, $0<q\leq \infty$, and $t<p$,
\item  $m-s_1+s_2=-d(1-\rho)\big(1/2-1/p \big)$, $2\leq p< \infty$, $0<t\leq \infty$, and $p<q$.
\end{enumerate}
\end{theorem}
\begin{theorem}\label{sharptheorem2}
Let $0<\rho<1$, $0<p,q,t\leq \infty$,  $s_1,s_2\in\mathbb{R}$, and $m\in\mathbb{R}$. Define $c_{m,\rho}(D)\in \mathcal{S}_{\rho,0}^{m}$ as in $(\ref{multiplierexample})$. Then
  \begin{equation*}
  \Vert c_{m,\rho}(D)\Vert_{B_p^{s_1,q}\to B_p^{s_2,t}}=\infty
  \end{equation*} if one of the followings holds.
\begin{enumerate}
\item $m-s_1+s_2>-d(1-\rho)\big|1/2-1/p \big|$,
\item $m-s_1+s_2=-d(1-\rho)\big|1/2-1/p \big|$ and $q>t$.
\end{enumerate}

\end{theorem}

This paper is organized as follows. We will give some preliminary results in Section \ref{property} and prove Theorem \ref{main}-\ref{mmmm} in Section \ref{basic}-\ref{infinityproof}.
 We construct some counter examples to prove Theorem \ref{sharptheorem1} and \ref{sharptheorem2} in Section \ref{sharpresult}.
 We give some remarks on pseudo-differential operators of type $(1,1)$ in Section \ref{type11}.
In Appendix we will discuss how our boundedness results can be extended to compound symbols in $\mathcal{S}_{\rho,\rho,\rho}^{m}$.

\section{\textbf{Preliminary results}}\label{property}

\subsection{Composition of pseudo-differential operators}
\hfill

In \cite[p14]{Ta}, symbolic calculus gives that if $T_j\in Op\mathcal{S}_{\rho_j,\delta_j}^{m_j}$ with $0\leq \delta_2<\rho_1\leq 1$ for $j=1,2$, then \begin{equation}\label{comcomcom}
T_1\circ T_2\in Op\mathcal{S}_{\rho,\delta}^{m_1+m_2}
\end{equation} where $\rho=\min\{\rho_1,\rho_2\}$ and $\delta=\max\{\delta_1,\delta_2\}$.
Thus, when we define 
\begin{equation*}
n_s(\xi):=\big(1+|\xi|^2\big)^{{s}/{2}}
\end{equation*}  we have 
\begin{equation*}
n_{s_1}(D)T_an_{s_2}(D)\in Op\mathcal{S}_{\rho,\rho}^{m+s_1+s_2}
\end{equation*}   and  \begin{equation*}
\big\Vert   T_af\big\Vert_{F_{p}^{s_2,2}} \approx \big\Vert n_{s_2}(D)T_an_{-s_1}(D)n_{s_1}(D)f\big\Vert_{h^p}\lesssim \big\Vert  n_{s_1}f \big\Vert_{h^p}\approx \big\Vert f  \big\Vert_{F_{p}^{s_1,2}}
\end{equation*}from the $h^p$ boundedness  with the hypothesis (\ref{ccexample}). 
Therefore
 \begin{equation}\label{finalcoral}
 T_a\in Op\mathcal{S}_{\rho,\rho}^{m} ~\text{maps}~ F_p^{s_1,2}~\text{into}~ F_p^{s_2,2}
 \end{equation} if (\ref{ccexample}) holds.
 
 This allows us to assume $s_1=s_2=0$ when we proceed for $0<\rho<1$ in this paper.

\subsection{$F$-spaces characterized by $L^p(l^q)$}\label{dual}\cite[p50]{Tr}
\hfill

One obtains an alternative description of $F_p^{s,q}$ spaces in the case $1<p,q<\infty$.
\begin{align*}
&F_p^{s,q}(\mathbb{R}^d)\\
&=\Big\{ f\in S'(\mathbb{R}^d):   \exists\{f_k\}_{k=0}^{\infty}\subset L^p(\mathbb{R}^d) ~\text{s.t.}~  f=\sum_{k=0}^{\infty}{\phi_k\ast f_k}  ~\text{in}~ S'(\mathbb{R}^d), ~\Vert \{2^{sk}f_k\}\Vert_{L^p(l^q)}<\infty   \Big\}\nonumber
\end{align*} and furthermore \begin{equation}\label{newdual}
\Vert f \Vert_{F_p^{s,q}}\approx \inf{\big\Vert \{2^{sk}f_k\} \big\Vert_{L^p(l^q)}}
\end{equation} where the infimum is taken over all of such admissible representations of $f$.

\subsection{Maximal inequalities}\cite{Pe, Tr}
\hfill

Denote by $\mathcal{M}$ the Hardy-Littlewood maximal operator and for $0<t<\infty$ let $\mathcal{M}_tu:=\big(\mathcal{M}(|u|^t)\big)^{1/t}$. For $r>0$ let $\mathcal{E}(r)$ denote the space of all distributions whose Fourier transforms are supported in $\{\xi:|\xi|\leq 2r\}$.
A crucial tool in theory of function spaces is a maximal operator introduced by Peetre \cite{Pe}.
For $r>0$ and $\sigma>0$ define 
\begin{equation*}
\mathcal{M}_{\sigma,r}u(x):=\sup_{y\in\mathbb{R}^d}{\dfrac{|u(x+y)|}{(1+r|y|)^{\sigma}}}.
\end{equation*} 
As shown in \cite{Pe}, one has the majorization \begin{equation*}
\mathcal{M}_{\sigma,r}u(x)\lesssim \mathcal{M}_tu(x)
\end{equation*} for all $\sigma\geq d/{t}$ if $u\in \mathcal{E}(r)$.
These estimates imply the following maximal inequality via the Fefferman-Stein inequality.
Suppose $0<p<\infty$ and $0<q\leq \infty$. Then for any sequence of positive numbers $\{r_k\}$ and $u_k\in \mathcal{E}(r_k)$ one has
\begin{equation}\label{max}
\Big\Vert  \Big(\sum_{k}{(\mathcal{M}_{\sigma,r_k}u_k)^q}\Big)^{1/{q}} \Big\Vert_{L^p} \lesssim  \Big\Vert \Big( \sum_{k}{|u_k|^q}  \Big)^{1/{q}}  \Big\Vert_{L^p} ~\text{for}~ \sigma>\max{\big\{d/p,d/q\big\}}.
\end{equation} 
The following is an immediate consequence of (\ref{max}).
Let $\zeta_0, \zeta \in S$ satisfy $Supp(\widehat{\zeta_0})\subset \{|\xi|\lesssim 1\}$ and $Supp(\widehat{\zeta})\subset \{1/c\leq |\xi|\leq c\}$ for some $c>0$, and set $\zeta_k(x):=2^{kd}\zeta(2^kx)$ for $k\geq 1$. Then\begin{equation}\label{tool}
\Big\Vert  \Big(   \sum_{k=0}^{\infty}{2^{skq }|\zeta_k\ast u|^q}  \Big)^{1/q}  \Big\Vert_{L^p} \lesssim \big\Vert    u\big\Vert_{F_p^{s,q}}.
\end{equation}

\subsection{$\varphi$-transform of $F$-spaces}\label{decomposition}\cite{Fr_Ja, Fr_Ja1, Fr_Ja2}

Let $\mathcal{D}$ be the set of all dyadic cubes in $\mathbb{R}^d$ and $\mathcal{D}_k$  the subset of $\mathcal{D}$ consisting of the cubes with sidelength $2^{-k}$. For $Q\in \mathcal{D}$ we denote the side length of $Q$ by $l(Q)$, lower left corner of $Q$ by $x_Q$, center of $Q$ by $c_Q$, and the characteristic function of $Q$ by $\chi_Q$.
For a sequence of complex numbers $b=\{b_Q\}_{{Q\in\mathcal{D}, l(Q)\leq 1}}$ we define 
\begin{equation*}
g^{s,q}(b)(x):=\Big(\sum_{{Q\in\mathcal{D}, l(Q)\leq 1}}{\big(|Q|^{-s/{d}-1/2}|b_Q|\chi_Q(x)\big)^q}\Big)^{1/q}
\end{equation*} and 
\begin{equation*}
\Vert b \Vert_{f_p^{s,q}}:=\big\Vert  g^{s,q}(b)  \big\Vert_{L^p}.
\end{equation*}
Furthermore for $c>0$ let  $\vartheta_0,\vartheta, \widetilde{\vartheta}_0, \widetilde{\vartheta} \in\mathcal{S}$ satisfy 
\begin{equation*}
Supp(\widehat{\vartheta}_0), Supp(\widehat{\widetilde{\vartheta}_0})\subset \{\xi : |\xi|\leq 2\}, 
\end{equation*}
\begin{equation*}
Supp(\widehat{\vartheta}), Supp(\widehat{\widetilde{\vartheta}})\subset \{\xi : 1/{2}\leq |\xi|\leq 2\}
\end{equation*}
\begin{equation*}
|\widehat{\vartheta_0}(\xi)|, |\widehat{\widetilde{\vartheta}_0}(\xi)| \geq c>0 ~\text{for}~ |\xi|\leq 5/{3}
\end{equation*}
\begin{equation*}
|\widehat{\vartheta}(\xi)|, |\widehat{\widetilde{\vartheta}}(\xi)| \geq c>0 ~\text{for}~ 3/4\leq |\xi|\leq 5/3
\end{equation*}
 and 
 \begin{equation*}
 \sum_{k=0}^{\infty}{\overline{\widetilde{\vartheta}_k(\xi)}\vartheta_k(\xi)}=1
 \end{equation*} where $\vartheta_k(x):=2^{kd}\vartheta(2^kx)$  and $\widetilde{\vartheta}_k(x):=2^{kd}\widetilde{\vartheta}(2^kx)$ for $k\geq 1$.
Then the norms in  $F_p^{s,q}$ can be characterized by the discrete $f_p^{s,q}$ norms.
Suppose $0<p<\infty$, $0<q\leq\infty$, $s\in\mathbb{R}$. 
 Every $f\in F_p^{s,q}$ can be decomposed as 
\begin{equation}\label{decomposition1}
f(x)=\sum_{{Q\in\mathcal{D}, l(Q)\leq 1}}{v_Q\vartheta^Q(x)} 
\end{equation} where $\vartheta^Q(x):=|Q|^{1/2}\vartheta_k(x-x_Q)$ for $l(Q)=2^{-k}$ and $v_Q:=<f,\widetilde{\vartheta}^Q>$.
Moreover in the case one has \begin{equation*}
\big\Vert v   \big\Vert_{f_p^{s,q}} \lesssim \big\Vert  f  \big\Vert_{F_p^{s,q}}.
\end{equation*}
The converse estimate also holds.  For any sequence $v=\{v_Q\}_{Q\in\mathcal{D}}$ of complex numbers satisfying $\big\Vert  v \big\Vert_{f_p^{s,q}}<\infty$,  
\begin{equation*}
f(x):=\sum_{{Q\in\mathcal{D}, l(Q)\leq 1}}{v_Q\vartheta^Q(x)}
\end{equation*} belongs to $F_p^{s,q}$ and 
\begin{equation}\label{decomposition2}
\big\Vert  f  \big\Vert_{F_p^{s,q}} \lesssim \big\Vert  v \big\Vert_{f_p^{s,q}}.
\end{equation}

\section{\textbf{Proof of Theorem \ref{main}}}\label{basic}

The first statement of Theorem \ref{main} follows simply from (\ref{finalcoral}), H\"older's inequality, and the embedding theorem $l^{p_1}\subset l^{p_2}$ for $p_1\leq p_2$. The second one is an  immediate consequence of (\ref{finalcoral}). 
Thus we shall prove just the third and the last statement. Assume $p\not=2$ and 
\begin{equation*}
m-s_1+s_2=-d(1-\rho)\big| 1/2-1/p  \big|.
\end{equation*}

It suffices to show that if $a(x,\xi)$ has compact support in $x$ variable, then
 \begin{equation*}
 \big\Vert T_af\big\Vert_{F_{p}^{s_2,t}}\leq C\Vert f\Vert_{F_{p}^{s_1,q}}
 \end{equation*} where a constant $C$ is independent of the compact support.
 To be specific we pick a smooth function $\gamma$ which is identically $1$ near a neighborhood of the origin and compactly supported.
Then \begin{equation}\label{supvar}
a^{\tau}(x,\xi):= a(x,\xi)\gamma( x/2^{\tau}), ~~\tau>0
\end{equation} belongs to $\mathcal{S}_{\rho,\rho}^m$ uniformly in $\tau$. 
For fixed $x\in\mathbb{R}^d$, 
\begin{equation*}
\big|  \phi_k\ast T_af(x)-\phi_k\ast T_{a^{\tau}}f(x)  \big|= \big| T_af\big( \phi_k(x-\cdot)\big(1-\gamma(\cdot/2^{\tau})\big)       \big) \big|.
\end{equation*}
Since $T_af \in S'(\mathbb{R}^d)$ this is bounded by
\begin{equation}\label{tag}
C\sum_{|\alpha|\leq N}{\big\Vert (1+|\cdot|)^{L}\partial^{\alpha}\big[ \phi_k(x-\cdot)\big(1-\gamma(\cdot/2^{\tau})\big)   \big]   \big\Vert_{L^{\infty}}}
\end{equation}   for some constants $C,L,N>0$, independent of $k$, $\tau$, and $x$.
For $\tau$ sufficiently large so that $|x| <100\cdot 2^{\tau}$, (\ref{tag}) is less than $C_{N,M} 2^{-k(M-N-d)}2^{-\tau(M-L)}$ for $M>N$ because $1-\gamma(y/2^{\tau})$ and any derivatives of it vanish on $\{y:|y|\lesssim 2^{\tau}\}$. This gives
\begin{equation*}
\big\Vert  \big\{2^{sk}\big(\phi_k\ast T_af(x)-\phi_k\ast T_{a^{\tau}}f(x)\big)\big\}  \big\Vert_{l^q(\mathbb{Z})}\lesssim_{N,M} 2^{-\tau(M-L)}
\end{equation*} for sufficiently large $M>0$, and thus one obtains
\begin{equation*}
\lim_{\tau \to\infty}{\big\Vert  \big\{ 2^{sk}\phi_k\ast T_{a^{\tau}}f(x)\big\}  \big\Vert_{l^q(\mathbb{Z})}}=\big\Vert \big\{2^{sk} \phi_k\ast T_af(x)   \big\} \big\Vert_{l^q(\mathbb{Z})}
\end{equation*} for fixed $x\in\mathbb{R}^d$.
Once $T_{a^{\tau}}$ is a bounded operator uniformly in $\tau$ then  by Fatou's lemma
\begin{equation*}
\big\Vert  T_af \big\Vert_{F_p^{s,q}}\leq \liminf_{\tau\to\infty}{\big\Vert  T_{a^{\tau}}f  \big\Vert_{F_p^{s,q}}}\lesssim \big\Vert f\big\Vert_{F_p^{s,q}}.
\end{equation*} 

Therefore one may assume $a(x,\xi)$ has a compact support in $x$ variable. The uniformity would be guaranteed because all of the estimates later will be made independently of the compact support. Indeed, one needs this compact support assumption just for doing integration by parts in order to obtain (\ref{sizeest}).

One may also assume $s_1=s_2=0$ due to the composition property of pseudo-differential operators in Section \ref{property}.

\subsection{  Paradifferential technique for Pseudo-differential operators }

The idea of our proof is based on the paradifferential technique, introduced by Bony \cite{Bo}. One splits $a\in\mathcal{S}_{\rho,\rho}^{m}$ into three symbols as follows.
Let 
\begin{equation}\label{defaa}
a_{j,k}(x,\xi)=
\begin{cases}
\phi_j\ast a(\cdot,\xi)(x)\widehat{\phi_k}(\xi) \quad &
\quad \quad j,k\geq 0
\\
0 \quad &
\quad  otherwise.
\end{cases}
\end{equation}
Then we decompose the symbol $a(x,\xi)$ as 
\begin{align*}
a(x,\xi)&=\sum_{k}{\sum_{j}{a_{j,k}(x,\xi)}}\\
    &=\sum_{j=3}^{\infty}{\sum_{k=0}^{j-3}{a_{j,k}(x,\xi)}}+\sum_{k=0}^{\infty}{\sum_{j=k-2}^{k+2}{a_{j,k}(x,\xi)}}+\sum_{k=3}^{\infty}{\sum_{j=0}^{k-3}{a_{j,k}(x,\xi)}}\\
    &=:a^{(1)}(x,\xi)+a^{(2)}(x,\xi)+a^{(3)}(x,\xi)
\end{align*}
and proceed by estimating each term separately.

Let $T^{(j)}$ be the pseudo-differential operators corresponding to each $a^{(j)}$.
Then our claim is that for any $s,m\in \mathbb{R}$ and $0<t\leq\infty$, $T^{(1)}$ and $T^{(2)}$ satisfy the estimates
\begin{equation}\label{12}
\big\Vert   T^{(1)}f  \big\Vert_{F_p^{0,t}}\lesssim \Vert  f \Vert_{F_p^{s,t}}
\end{equation}  and 
\begin{equation}\label{34}
\big\Vert   T^{(2)}f  \big\Vert_{F_p^{0,t}}\lesssim \Vert  f \Vert_{F_p^{s,t}},
\end{equation}  which, of course, imply both operators map $F_p^{0,q}$ into $F_p^{0,t}$ for $0<q\leq\infty$.

\subsection{Proof of (\ref{12})}
It follows in the same way as in \cite{Joh}, which is based on the following two lemmas.
\begin{lemma}\cite[Lemma 2.1]{Joh}\label{Yaineq}, \cite[Theorem 3.6]{Ya}
Let $A>0$, $s \in \mathbb{R}$, $0<p<\infty$,  $0<q\leq \infty$ and $\{u_j\}_{j=0}^{\infty}$ be a sequence in $S'(\mathbb{R}^d)$ satisfying 
$Supp(\widehat{u_0})\subset \{\xi: |\xi|\leq A\}$ and $Supp(\widehat{u_j})\subset \{\xi:2^{j}A^{-1} \leq |\xi|\leq 2^{j}A\}$ for $j\geq 1$. 
\begin{enumerate}
\item If   $\big\Vert \{2^{sj}u_j\}_{j=0}^{\infty}\big\Vert_{l^q(L^p)}<\infty$, then $ \sum_{j=0}^{\infty}{u_j}$ converges in $S'(\mathbb{R}^d)$ to some $u\in B_{p}^{s,q}$ and
 \begin{equation*}
 \Vert   u\Vert_{{B}_p^{s,q}}\lesssim_A  \big\Vert \{2^{sj}u_j\}_{j=0}^{\infty}\big\Vert_{l^q(L^p)}.   
 \end{equation*}
\item If $ \big\Vert  \{2^{sj}u_j\}_{j=0}^{\infty}   \big\Vert_{L^p(l^q)}<\infty$, then $\sum_{j=0}^{\infty}{u_j}$ converges in $S'(\mathbb{R}^d)$ to some $u\in F_{p}^{s,q}$ and
 \begin{equation*}
 \Vert   u\Vert_{{F}_p^{s,q}}\lesssim_A \big\Vert  \{2^{sj}u_j\}_{j=0}^{\infty}   \big\Vert_{L^p(l^q)}.
 \end{equation*}
 \end{enumerate}
\end{lemma}
\begin{lemma}\cite{Joh, Ma}\label{Maineq} 
Let $A>0$ and let $v\in\mathcal{S}'(\mathbb{R}^d)$ and $b(x,\xi)$ be a function on $\mathbb{R}^d\times\mathbb{R}^d$ such that $Supp(\widehat{v})\subset \{\xi : |\xi|\leq 2^kA\}$ and $b(x,\xi)=0$ for $|\xi|>2^kA$.
Then for all $0<r\leq 1$, 
\begin{equation*}
|T_bv(x)|\lesssim_A  \big\Vert  b(x,2^k\cdot)\big\Vert_{W^{c^{(r,d)},1}}\mathcal{M}_rv(x)
\end{equation*} where  $c^{(r,d)}$ is the smallest integer greater than $d/r$ and $W^{c^{(r,d)} ,1}$ is the Sobolev space with $ \Vert  u  \Vert_{W^{c^{(r,d)},1}}=\sum_{|\alpha|\leq c^{(r,d)} }{\Vert \partial^{\alpha}u \Vert_{L^1}}$.
 \end{lemma}

Observe that  
\begin{equation*}
T_{a_{j,k}}f(x)=T_{a_{j,k}}(f\ast \widetilde{\phi}_k)(x)
\end{equation*} where $\widetilde{\phi}_k=\phi_{k-1}+\phi_k+\phi_{k+1}$ and
the Fourier transform of $ u_j := \sum_{k=0}^{j-3}{T_{a_{j,k}}f      }$ is supported in the annulus $\{\eta : 2^{j-2}\leq|\eta|\leq 2^{j+2}\}$.
By  Lemma \ref{Yaineq} one has
\begin{equation}\label{est1}
\Vert T^{(1)}f\Vert_{{F}_{p}^{0,t}}  \lesssim \Big\Vert  \Big( \sum_{j=3}^{\infty}{\Big|\sum_{k=0}^{j-3}{ T_{a_{j,k}}(f\ast \widetilde{\phi}_k)             }\Big|^{t}}          \Big)^{1/{t}}   \Big\Vert_{L^p}
\end{equation} 
and then apply Lemma \ref{Maineq} to get
\begin{equation*}
\big| T_{a_{j,k}}(f\ast \widetilde{\phi}_k)(x)\big|\lesssim \big\Vert   a_{j,k}(x,2^k\cdot)\big\Vert_{W^{c^{(r,d)},1}}\mathcal{M}_r(f\ast \widetilde{\phi}_k)(x)
\end{equation*} for $0<r<\min{\{1,p,t\}}$.
For each multi-index $\alpha_1$ and $|\xi|\approx 1$, using the cancellation condition of $\phi_j$,
we see that
\begin{align*}
&\big| \phi_j\ast \partial_{\xi}^{\alpha_1}a(\cdot,2^k\xi)(x)         \big| \\
&\lesssim \sum_{|\beta|=N}{\int_{\mathbb{R}^d}{|\phi_j(y)||y|^N     \int_0^1 \dots  \int_0^1{t_1^{N-1}\dots t_{N-1}\big|\partial_x^{\beta}\partial_{\xi}^{\alpha_1}a(x-t_1\dots t_Ny,2^k\xi)\big|    }dt_N\dots d t_1    }dy}\\
      &\lesssim_N 2^{k(m-\rho|\alpha_1|+\rho N)-jN}
\end{align*}  for any positive integer $N$. Here we used Taylor's theorem
\begin{equation*}
g(x-y)=\sum_{|\beta|=1}\int_0^1{\partial^{\beta}g(x-ty)}dt(-y)^{\beta}
\end{equation*}
$N$ times for the first inequality.
Then this proves
\begin{align*}
 \big\Vert   a_{j,k}(x,2^k\cdot)\big\Vert_{W^{c^{(r,d)},1}}&\lesssim \sum_{|\alpha|\leq c^{(r,d)}}{\int_{|\xi|\approx 1}{\Big| \partial_{\xi}^{\alpha}\big(\phi_j\ast a(\cdot,2^k\xi)(x)\widehat{\phi}(\xi)\big)   \Big|}d\xi}\\
   &\lesssim 2^{k(m+c^{(r,d)}(1-\rho)+\rho N)-jN}.
\end{align*} Therefore 
\begin{equation}\label{firstkey}
\big| T_{a_{j,k}}(f\ast\widetilde{\phi}_k)(x)  \big|\lesssim 2^{k(m+c^{(r,d)}(1-\rho)+\rho N)-jN} \mathcal{M}_r(f\ast \widetilde{\phi}_k)(x)
\end{equation}
and finally one has
\begin{align*}
\Vert T^{(1)}f\Vert_{{F}_p^{0,t}}   &\lesssim\Big\Vert  \Big( \sum_{j=3}^{\infty}{2^{-N jt}\Big(\sum_{k=0}^{j-3}{ 2^{k(m+c^{(r,d)}(1-\rho)+\rho N)}\mathcal{M}_r\big(f\ast \widetilde{\phi}_k\big)             }\Big)^{t}}          \Big)^{1/{t}}   \Big\Vert_{L^p}\\
    &\lesssim \Big\Vert   \Big( \sum_{k=0}^{\infty}{2^{kt(m+c^{(r,d)}(1-\rho)+\rho N-\epsilon)}\big(\mathcal{M}_r(f\ast\widetilde{\phi}_k)\big)^t \sum_{j=k+3}^{\infty}{2^{-jt(N-\epsilon)}}            }   \Big)^{1/t}  \Big\Vert_{L^p}\\
    &\lesssim \Big\Vert   \Big(  \sum_{k=0}^{\infty}{2^{kt(m+c^{(r,d)}(1-\rho)-N(1-\rho))}\big|f\ast \widetilde{\phi}_k\big|^{t}}      \Big)^{{1}/{t}} \Big\Vert_{L^p}\\
    &\lesssim  \Vert f\Vert_{{F}_{p}^{m+c^{(r,d)}(1-\rho)-N(1-\rho),t}} 
\end{align*} for $N$ sufficiently large and $0<\epsilon <N$ where the third inequality follows from the $L^p(l^t)$ boundedness of maximal operator $\mathcal{M}_r$ with $0<r<p,t$ and the last one is from (\ref{tool}).
For all $s\in\mathbb{R}$ we choose $N$ sufficiently large so that 
\begin{equation*}
m+c^{(r,d)}(1-\rho)-N(1-\rho)<s.
\end{equation*} This ends the proof of (\ref{12}).

\subsection{Proof of (\ref{34})}

Now we consider the operator $T^{(2)}$.
By setting 
\begin{equation*}
 a_k(z,\eta):=\sum_{j=k-2}^{k+2}{a_{j,k}(z,\eta)}=\Phi_k\ast a(\cdot,\eta)(z)\widehat{\phi_k}(\eta)
 \end{equation*} where $ \Phi_k:=\sum_{j=k-2}^{k+2}{\phi_j}$, we express $T^{(2)}$ as 
\begin{equation}\label{secondoperator}
T^{(2)}= \sum_{k=0}^{9}{T_{a_k}}+  \sum_{k=10}^{\infty}{T_{a_k}}.
\end{equation}
The finite sum of operators clearly satisfies (\ref{34}). 
For each $k\geq 10$ the kernel of $T_{a_k}$ is 
 \begin{align}\label{kernelpart}
 K_k(x,y)&=\int_{\mathbb{R}^d}{a_k(x,\xi)e^{2\pi i\langle x-y,\xi\rangle }}d\xi\nonumber\\
             &=\int_{\mathbb{R}^d\times\mathbb{R}^d\times\mathbb{R}^d}{{{  a(z,\xi)\widehat{\Phi}_k(\eta)\widehat{\phi}_k(\xi)e^{-2\pi i\langle z,\eta\rangle }e^{2\pi i\langle x,\eta\rangle }e^{2\pi i\langle x-y,\xi\rangle }     }dz}d\eta}d\xi .
 \end{align} Observe that in the integral the variables $\xi$ and $\eta$ live in $|\xi|\approx 2^k$ and $|\eta|\approx 2^k$. Then it has the size estimate 
\begin{equation}\label{sizeest}
\big|K_k(x,y)\big|\lesssim_{J,N} 2^{-J k}\dfrac{1}{(1+|x-y|)^N}
\end{equation} for any $J>0$ and $N>0$.

The idea to get (\ref{sizeest}) is to apply a technique of oscillatory integrals by integrating by parts with respect to each variable. First, we perform this with respect to the $z$-variable (we could do this due to our compact support hypothesis) and then carry out a similar process on the $\eta$-variable.
Now (\ref{kernelpart}) is dominated by
\begin{equation*}
C_M2^{-k(2M-d)}\int_{\mathbb{R}^d}{    \Big|\int_{|\xi|\approx 2^k}{ (I-\Delta_z)^Ma(z,\xi)e^{2\pi i\langle x-y,\xi\rangle }\widehat{\phi}_k(\xi) }d\xi\Big|  \dfrac{1}{(1+|x-z|)^{2d}}}dz.
\end{equation*} for any $M>0$.
If $|x-y|\leq 1$, then \begin{equation*}
\Big|\int_{|\xi|\approx 2^k}{ (I-\Delta_z)^Ma(z,\xi)e^{2\pi i\langle x-y,\xi\rangle } \widehat{\phi}_k(\xi)}d\xi\Big|
\lesssim \int_{|\xi|\approx 2^k}{\big(1+|\xi|\big)^{m+2\rho M}}d\xi \lesssim 2^{k(m+2\rho M +d)}.
\end{equation*}
If $|x-y|>1$, then we do integration by parts in $\xi$-variable  to get
\begin{align*}
\Big|\int_{|\xi|\approx 2^k}{ (I-\Delta_z)^Ma(z,\xi)e^{2\pi i\langle x-y,\xi\rangle } \widehat{\phi}_k(\xi)}d\xi\Big|&\lesssim_{N,M} \dfrac{1}{|x-y|^N}\int_{|\xi|\approx 2^k}{\big(1+|\xi|\big)^{m-\rho N+2\rho M}}d\xi\\
&\lesssim_{N,M} 2^{k(m-\rho N+2\rho M+d)}\dfrac{1}{|x-y|^N}.
\end{align*}
These yield (\ref{sizeest})  by choosing $M$ and $N$ sufficiently large.

Since $\widehat{T_{a_k}f}$ is supported in $\{|\eta|\leq 2^{k+4}\}$, we have 
\begin{align}
\Big\Vert    \sum_{k=10}^{\infty}{T_{a_k}}f  \Big\Vert_{F_{p}^{0,t}}&\leq \Big\Vert \Big( \sum_{j=0}^{\infty}{ \big( \sum_{k=10}^{\infty }{|\phi_j\ast T_{a_k}(f\ast \widetilde{\phi}_k)|       }   \big)^{t}}       \Big)^{{1}/{t}}    \Big\Vert_{L^{p}}\nonumber\\
  &\lesssim \Big\Vert \Big( \sum_{j=0}^{13}{ \big( \sum_{k=10}^{\infty }{|\phi_j\ast T_{a_k}(f\ast \widetilde{\phi}_k)|       }   \big)^{t}}       \Big)^{{1}/{t}}    \Big\Vert_{L^{p}}\label{secondmain1}\\
  &\relphantom{=} +\Big\Vert \Big( \sum_{j=14}^{\infty}{ \big( \sum_{k=j-4}^{\infty }{|\phi_j\ast T_{a_k}(f\ast \widetilde{\phi}_k)|       }   \big)^{t}}       \Big)^{{1}/{t}}    \Big\Vert_{L^{p}}\label{secondmain2}.
\end{align}

Let $\sigma>\max{\big\{{d}/{p},{d}/{t}\big\}}$ and choose $N>\sigma$.
Then \begin{align*}
\big|\phi_j\ast T_{a_k}(f\ast \widetilde{\phi}_k)(x)\big|&\leq \int_{\mathbb{R}^d}{\int_{\mathbb{R}^d}{ |\phi_j(x-z)||K_k(z,y)||f\ast \widetilde{\phi}_k(y)|  }dy}dz   \\
   &\lesssim \mathcal{M}_{\sigma,2^k}{(f\ast \widetilde{\phi}_k)}(x)\int_{\mathbb{R}^d}{\int_{\mathbb{R}^d}{ \big(1+2^k|x-y|\big)^{\sigma}\big| \phi_j(x-z)\big| \big| K_k(z,y)\big|   }dy}dz\\
   &\lesssim 2^{-k(J-\sigma)} \mathcal{M}_{\sigma,2^k}{(f\ast \widetilde{\phi}_k)}(x)
\end{align*} by the size estimate (\ref{sizeest}).
Thus for $\epsilon>0$
\begin{align*}
(\ref{secondmain1})&\lesssim \Big\Vert  \sum_{k=10}^{\infty}{2^{-k(J-\sigma)} \mathcal{M}_{\sigma,2^k}(f\ast \widetilde{\phi}_k)     }  \Big\Vert_{L^p}\\
&\lesssim \Big\Vert  \Big(  \sum_{k=10}^{\infty}{2^{-kt(J-\sigma-\epsilon)} \big(\mathcal{M}_{\sigma,2^k}(f\ast \widetilde{\phi}_k)\big)^t  }    \Big)^{1/t}  \Big\Vert_{L^p}\\
       &\lesssim \Big\Vert  \Big(  \sum_{k=10}^{\infty}{2^{-kt(J-\sigma-\epsilon)} \big|f\ast \widetilde{\phi}_k\big|^t  }    \Big)^{1/t}  \Big\Vert_{L^p}\lesssim\Vert f\Vert_{F_{p}^{-(J-\sigma-\epsilon),t}}
\end{align*} by (\ref{max}) and (\ref{tool}).
Similarly we also get 
\begin{equation*}
(\ref{secondmain2}) \lesssim \Vert f\Vert_{F_{p}^{-(J-\sigma-\epsilon),t}}.
\end{equation*} 
Then (\ref{34}) follows by choosing $J$ satisfying $-(J-\sigma-\epsilon)<s.$

\subsection{Boundedness of $T^{(3)}$}
We write a symbol $a^{(3)}\in\mathcal{S}_{\rho,\rho}^{m}$ as 
\begin{equation*}
a^{(3)}(x,\xi)=\sum_{k=3}^{\infty}{b_k(x,\xi)}
\end{equation*} where $b_k(x,\xi):=\sum_{j=0}^{k-3}{a_{j,k}(x,\xi)}.$

Since  $\widehat{T_{b_k}u}$ is supported in $\{\xi : 2^{k-2}\leq |\xi|\leq 2^{k+2}\}$ it follows that for $0<r<\infty$
\begin{equation*}
\Vert T_{b_k}u \Vert_{L^r} \lesssim \big\Vert   T_{b_k}u \big\Vert_{h^r}.
\end{equation*} In addition, 
\begin{equation*}
b_k(x,\xi)=\Big(\sum_{j=0}^{k-3}{\phi_j}\Big)\ast a(\cdot,\xi)(x)
\end{equation*} is a $\mathcal{S}^{m}_{\rho,\rho}$ symbol with a constant which is independent of $k$, and thus
 for  $u\in S'(\mathbb{R}^d)$,
\begin{equation*}
\big\Vert   T_{b_k}u \big\Vert_{h^r} \leq C_r 2^{k( d(1-\rho)( | {1}/{r}-{1}/{2}  |-| {1}/{p}-{1}/{2} |   ) )}\big\Vert \phi_k\ast u \big\Vert_{h^r}
\end{equation*} 
by (\ref{finalcoral}), which implies
\begin{equation}\label{pqr}
\big\Vert   T_{b_k}u \big\Vert_{L^r} \lesssim C_r 2^{k( d(1-\rho)( | {1}/{r}-{1}/{2}  |-| {1}/{p}-{1}/{2} |   ) )}\big\Vert \phi_k\ast u \big\Vert_{h^r}.
\end{equation} 

By the same reasoning as (\ref{est1}) one has
\begin{equation*}
\big\Vert   T^{(3)}f  \big\Vert_{F_p^{0,t}} \lesssim \Big\Vert  \Big(   \sum_{k=3}^{\infty}{|T_{b_k}f|^t}           \Big)^{{1}/{t}}   \Big\Vert_{L^p}.
\end{equation*}
Therefore in order to conclude the proof of Theorem \ref{main} it is enough to show that
\begin{equation}\label{goal}
\Big\Vert  \Big(   \sum_{k=3}^{\infty}{|T_{b_k}f|^t}           \Big)^{{1}/{t}}   \Big\Vert_{L^p} \lesssim \big\Vert  f \big\Vert_{F_p^{0,q}}.
\end{equation} 
We shall study just the cases $0<p\leq 1$ and $2<p<\infty$.
Then since the adjoint operator $(T_a)^*$ is also in $Op\mathcal{S}_{\rho,\rho}^m$ the case $1<p<2$ can be derived via duality.  Indeed, for $1<p<2$
\begin{align*}
\Vert T_af\Vert_{F_p^{0,p}}&= \big\Vert \big\{\phi_k\ast T_af\big\}\big\Vert_{L^p(l^p)}=\sup_{\Vert \{g_k\} \Vert_{L^{p'}(l^{p'})}\leq 1}{\Big| \int_{\mathbb{R}^d}{\sum_{k=0}^{\infty}{\phi_k\ast T_af(x)g_k(x)}}dx  \Big|}\\
                      &= \sup_{\Vert g_k \Vert_{L^{p'}(l^{p'})}\leq 1}{\Big| \int_{\mathbb{R}^d}{\sum_{j=0}^{\infty}{\phi_j\ast f(x)\widetilde{\phi}_j\ast\Big( T_a^*\Big(\sum_{k=0}^{\infty}{\phi_k\ast g_k}\Big)  \Big)(x)     }}dx    \Big|   }\\
                      &\lesssim \Vert f\Vert_{F_p^{0,\infty}}\sup_{\Vert \{g_k\} \Vert_{L^{p'}(l^{p'})}\leq 1}{\Big\Vert  T_a^*\Big(\sum_{j=0}^{\infty}{\phi_j\ast g_j}\Big)\Big\Vert_{F_{p'}^{0,1}}}\\
                      &\lesssim \Vert f\Vert_{F_p^{0,\infty}}\sup_{ \Vert \{g_k\} \Vert_{L^{p'}(l^{p'})}\leq 1}{\Big\Vert \sum_{k=0}^{\infty}{\phi_k\ast g_k}  \Big\Vert_{F_{p'}^{0,p'}}}\leq\Vert f \Vert_{F_{p}^{0,\infty}}
\end{align*} where the last inequality holds due to (\ref{newdual}). Here everything makes sense and one may use (\ref{newdual}) because the infinite sum of $\phi_k\ast g_k$ belongs to $S'$ due to Lemma \ref{Yaineq} with the estimate 
\begin{equation*}
\Big\Vert \Big(\sum_{k=0}^{\infty}{|\phi_k\ast g_k|^{p'}}  \Big)^{1/p'} \Big\Vert_{L^{p'}}\leq \Vert g_k \Vert_{L^{p'}(l^{p'})}\leq 1.
\end{equation*}

\subsection{   Proof of (\ref{goal}); the case $0<p\leq 1 $ }
One needs to prove (\ref{goal}) with $t=p$ and $q=\infty$.
In \cite{Pa_So} P\"aiv\"arinta and Somersalo use the atomic decomposition of the local hardy space for $0<p\leq 1$. It is therefore natural to use an adaption of the atomic decomposition of Triebel-Lizorkin spaces. Alternatively, one can characterize Triebel-Lizorkin spaces $F_p^{s,q}$ with the associated sequence spaces $f_p^{s,q}$ via the Frazier and Jawerth $\varphi$-transform, and then use atomic decomposition for the sequence spaces. We follow the latter approach and recall definitions.

\begin{definition}
Let $0<p\leq 1$, $0< q\leq \infty$, and $s\in\mathbb{R}$. A sequence of complex numbers $r=\{r_Q\}_{{Q\in\mathcal{D}, l(Q)\leq 1}}$ is called an $\infty$-atom for $f_p^{s,q}$ if there exists a dyadic cube $Q_0$ such that 
\begin{equation*}
r_Q=0 \quad \text{if}\quad Q \not\subset Q_0
\end{equation*}
 and \begin{equation}\label{infdef}
\big\Vert  g^{s,q}(r)  \big\Vert_{L^{\infty}}\leq |Q_0|^{-{1}/{p}}.
\end{equation}

\end{definition}

The following atomic decomposition of $f_p^{s,q}$ substitutes the atomic decomposition of $h^p$.
\begin{lemma}\label{decomhardy}\cite{Fr_Ja2}, \cite[6.6.3]{Gr}
Suppose $0<p\leq 1$, $p\leq q\leq\infty$, and $b=\{b_Q\}_{Q\in\mathcal{D},l(Q)\leq1}\in f_p^{s,q}$. Then there exist $C_{d,p,q}>0$, a sequence of scalars $\{\lambda_j\}$, and a sequence of $\infty$-atoms $r_j=\{r_{j,Q}\}_{{Q\in\mathcal{D}, l(Q)\leq 1}}$ for $f_p^{s,q}$ such that 
\begin{equation*}
b=\{b_Q\}=\sum_{j=1}^{\infty}{\lambda_j\{r_{j,Q}\}}=\sum_{j=1}^{\infty}{\lambda_j r_j}
\end{equation*} and such that 
\begin{equation*}
\Big(\sum_{j=1}^{\infty}{|\lambda_j|^p}\Big)^{{1}/{p}}\leq C_{d,p,q}\big\Vert   b\big\Vert_{f_{p}^{s,q}}.
\end{equation*}
Moreoever, 
\begin{equation*}
\big\Vert  b  \big\Vert_{f_p^{s,q}}\approx \inf{\Big\{ \Big(\sum_{j=1}^{\infty}{|\lambda_j|^p}\Big)^{{1}/{p}}   :  b=\sum_{j=1}^{\infty}{\lambda_j r_j} ,~ r_j ~\text{is an $\infty$-atom for $f_p^{s,q}$}    \Big\}}.
\end{equation*}

\end{lemma}

By (\ref{decomposition1}) and Lemma \ref{decomhardy}, $f\in F_p^{0,\infty}$ can be decomposed with $\{b_Q\}_{\substack{Q\in\mathcal{D}\\l(Q)\leq 1}}\in f_p^{0,\infty}$ and there exist a sequence of scalars $\{\lambda_j\}$ and a sequence of $\infty$-atoms $\{r_{j,Q}\}$ for $f_p^{0,\infty}$ such that
\begin{equation*}
f(x)=\sum_{{Q\in\mathcal{D}, l(Q)\leq 1}}{b_Q\vartheta^Q(x)}=\sum_{j=1}^{\infty}{\lambda_j \sum_{{Q\in\mathcal{D}, l(Q)\leq 1}}{ r_{j,Q}\vartheta^Q (x)  }}.
\end{equation*}
Then
\begin{align}\label{sup1}
\big\Vert   T^{(3)}f \big\Vert_{F_p^{0,p}} &\lesssim \Big\Vert  \Big(\sum_{k=3}^{\infty}{|T_{b_k}f|^p}\Big)^{{1}/{p}}   \Big\Vert_{L^p} \nonumber  \\
  &= \Big\Vert  \Big(\sum_{k=3}^{\infty}{\Big|   \sum_{j=1}^{\infty}{\lambda_j T_{b_k}\Big(\sum_{{Q\in\mathcal{D}_k, l(Q)\leq 1}}{r_{j,Q}\vartheta^Q}\Big)}   \Big|^p}\Big)^{{1}/{p}}    \Big\Vert_{L^p}\nonumber\\
  &\leq\Big( \sum_{j=1}^{\infty}{|\lambda_j|^p}\int_{\mathbb{R}^d}{ \sum_{k=3}^{\infty}{\Big| T_{b_k}\Big(    \sum_{{Q\in\mathcal{D}_k, l(Q)\leq 1}}{r_{j,Q}\vartheta^Q}          \Big)  \Big|^p}   }dx\Big)^{{1}/{p}}\nonumber\\
  &\lesssim \big( \sum_{j=1}^{\infty}{|\lambda_j|^p} \big)^{{1}/{p}}\sup_{j}{\Big\{  \Big(  \sum_{k=3}^{\infty}{\Big\Vert T_{b_k}\Big( \sum_{{Q\in\mathcal{D}_k, l(Q)\leq 1}}{r_{j,Q}\vartheta^Q}   \Big)  \Big\Vert_{L^p}^p  }     \Big)^{{1}/{p}}\Big\}}  
\end{align} by using triangle inequality for $p\leq1$ and  $l^p\subset l^1$.
Since 
\begin{equation*}
\Big( \sum_{j=1}^{\infty}{|\lambda_j|^p} \Big)^{{1}/{p}}\lesssim \Vert f \Vert_{F_p^{0,\infty}},
\end{equation*} it suffices to show the supremum in (\ref{sup1}) is bounded by a constant.

Let $Q_0$ be any dyadic cubes with side length $2^{-\mu}$ and $r_Q$ be an $\infty$-atom for $f_p^{0,\infty}$ with $Q_0$ and define 
\begin{equation*} 
R_{Q_0,k}(x):=\sum_{\substack{Q\in\mathcal{D}_k, Q\subset Q_0\\l(Q)\leq 1}}{r_Q\vartheta^Q(x)}.
\end{equation*} 
Then one obtains the desired result by showing
\begin{equation}\label{gggoal}
 \Big( \sum_{k=3}^{\infty}{\big\Vert T_{b_k}R_{Q_0,k}\Vert_{L^p}^p} \Big)^{{1}/{p}}\lesssim 1 \quad \text{uniformly in $Q_0$}.
\end{equation} 

Note that by (\ref{infdef}) one has \begin{equation}\label{rcondition}
|r_Q|  \leq |Q|^{1/2}|Q_0|^{-{1}/{p}}.
\end{equation}

Let $Q_0^*$ be a dilate of $Q_0$ by a factor of 10 and ${Q_0}^{**}$ by a factor of $100\sqrt{d}$.
Furthermore we denote by $Q_0^{\rho}$ a dilate of $Q_0$ with side $10l(Q_0)^{\rho}$, by $\widetilde{Q_0^{\rho}}$ with side $100\sqrt{d}l(Q_0)^{\rho}$. 

We first consider the case $l(Q_0)\leq 2^{-3}$ ( i.e. $\mu \geq 3$ ).
The condition $Q\subset Q_0$ in the definition of $R_{Q_0,k}$ ensures that $R_{Q_0,k}$ vanishes unless $\mu\leq k$. Hence the summation can be taken over  $k\geq \mu $ in (\ref{gggoal}).
For each $k$ we split the range of the integral into two parts, $\widetilde{Q_0^{\rho}}$ and $(\widetilde{Q_0^{\rho}})^c$. That is,
\begin{equation*}
 \Big( \sum_{k=3}^{\infty}{\big\Vert T_{b_k}R_{Q_0,k}\Vert_{L^p}^p} \Big)^{{1}/{p}}
\lesssim \Big( \sum_{k=\mu}^{\infty}{\big\Vert T_{b_k}R_{Q_0,k}\big\Vert_{L^p(\widetilde{Q_0^{\rho}})}^p} \Big)^{{1}/{p}}+\Big( \sum_{k=\mu}^{\infty}{\big\Vert T_{b_k}R_{Q_0,k}\big\Vert_{L^p(\widetilde{Q_0^{\rho}}^c)}^p} \Big)^{{1}/{p}}.
\end{equation*}

For the term corresponding to $\widetilde{Q_0^{\rho}}$ we use H\"older's inequality, (\ref{pqr}), (\ref{decomposition2}), and (\ref{rcondition}). Then it follows that for $p<r$
\begin{align*}
&\big\Vert T_{b_k}R_{Q_0,k}\big\Vert_{L^p(\widetilde{Q_0^{\rho}})}\leq|\widetilde{Q_0^{\rho}}|^{{1}/{p}-{1}/{r}}\big\Vert T_{b_k}R_{Q_0,k}  \big\Vert_{L^r}\\
    &\lesssim 2^{-\mu\rho d({1}/{p}-{1}/{r})}2^{-kd(1-\rho)(1/p-1/r)}\Vert  R_{Q_0,k} \Vert_{h^r}\\
    &\lesssim  2^{-\mu\rho d({1}/{p}-{1}/{r})}2^{-kd(1-\rho)(1/p-1/r)}  \Big\Vert \Big(\sum_{Q\in\mathcal{D}_k,Q\subset Q_0}{(|r_Q||Q|^{-{1}/{2}}\chi_Q)^2}\Big)^{{1}/{2}}\Big\Vert_{L^r}\\
    &\leq 2^{\mu d(1-\rho)({1}/{p}-{1}/{r})}2^{-kd(1-\rho)(1/p-1/r)}.
\end{align*}
 This proves
\begin{equation*}
\Big( \sum_{k=\mu}^{\infty}{\big\Vert T_{b_k}R_{Q_0,k}\Vert_{L^p(\widetilde{Q_0^{\rho}})}^p} \Big)^{{1}/{p}} \lesssim 1
\end{equation*}uniformly in $\mu$.

For the latter one we split it into 
\begin{equation}\label{s2}
\Big( \sum_{k=\mu}^{\infty}{\big\Vert T_{b_k}(\chi_{(Q_0^*)^c}R_{Q_0,k})\Vert_{L^p(\widetilde{Q_0^{\rho}}^c)}^p} \Big)^{{1}/{p}}
\end{equation} and 
\begin{equation}\label{s1}
\Big( \sum_{k=\mu}^{\infty}{\big\Vert T_{b_k}(\chi_{Q_{0}^*}R_{Q_0,k})\Vert_{L^p(\widetilde{Q_0^{\rho}}^c)}^p} \Big)^{{1}/{p}}
\end{equation} and our claim is that each part can be controlled by a constant independent of $\mu$.

By applying (\ref{pqr}) we write
 \begin{equation*}
(\ref{s2})\lesssim \Big(\sum_{k=\mu}^{\infty}{\big\Vert  \phi_k\ast(\chi_{(Q_0^*)^c}R_{Q_0,k}) \big\Vert_{h^p}^p}\Big)^{{1}/{p}}
\end{equation*} and the Fourier support of $\phi_k$ implies the estimate
\begin{align*}
& \big\Vert  \phi_k\ast\big(\chi_{(Q_0^*)^c}R_{Q_0,k}\big) \big\Vert_{h^p}\\
&\lesssim \big\Vert  \phi^{(-1)}_k\ast\big(\chi_{(Q_0^*)^c}R_{Q_0,k}\big) \big\Vert_{L^p}+\big\Vert  \phi^{(0)}_k\ast\big(\chi_{(Q_0^*)^c}R_{Q_0,k}\big) \big\Vert_{L^p}+\big\Vert  \phi^{(1)}_k\ast\big(\chi_{(Q_0^*)^c}R_{Q_0,k}\big) \big\Vert_{L^p}
\end{align*} where $\phi_k^{(-1)}=\phi_{k-1}\ast\phi_{k}$, $\phi_k^{(0)}=\phi_{k}\ast\phi_{k}$, and $\phi_k^{(1)}=\phi_{k+1}\ast\phi_{k}$.
For each $j\in\{-1,0,1\}$,
\begin{align*}
&\big\Vert  \phi^{(j)}_k\ast\big(\chi_{(Q_0^*)^c}R_{Q_0,k}\big) \big\Vert_{L^p}^p\leq\int_{\mathbb{R}^d}{\Big(\int_{y\in{(Q_0^*)}^c}{\big|\phi^{(j)}_k(x-y)\big|\sum_{{Q\in\mathcal{D}_k, Q\subset Q_0}}{|r_Q||\vartheta^Q(y)|}}dy\Big)^p}dx\\
&\lesssim 2^{-kp(N-{d}/{2})}\int_{\mathbb{R}^d}{\Big(\int_{y\in{(Q_0^*)}^c}{|\phi^{(j)}_k(x-y)|\sum_{{Q\in\mathcal{D}_k, Q\subset Q_0}}{\frac{|r_Q|}{|y-x_Q|^N}}}dy\Big)^p}dx\\
&\leq 2^{-kp(N-{d}/{2})}\Big(\sum_{{Q\in\mathcal{D}_k, Q\subset Q_0}}{|r_Q|}\Big)^{p}\int_{\mathbb{R}^d}{\Big(\int_{y\in{(Q_0^*)}^c}{\frac{|\phi^{(j)}_k(x-y)|}{|y-c_{Q_0}|^N}}dy\Big)^p}dx\\
&\leq 2^{-kp(N-d)}2^{\mu dp({1}/{p}-1)}   \int_{\mathbb{R}^d}{\Big(\int_{y\in{(Q_0^*)}^c}{\frac{|\phi^{(j)}_k(x-y)|}{|y-c_{Q_0}|^N}}dy\Big)^p}dx\\
&\lesssim 2^{-kp(N-2d+{d}/{p})}2^{\mu dp({1}/{p}-1)} \Big(  \int_{y\in{(Q_0^*)}^c}{ \dfrac{1}{|y-c_{Q_0}|^N}  \int_{\mathbb{R}^d}{(1+2^k|x-c_{Q_0}|)^{{L}/{p}}|\phi^{(j)}_k(x-y)|}dx}dy\Big)^p\\
&\lesssim 2^{-kp(N-2d+{d}/{p}-L/p)}2^{\mu dp({1}/{p}-1)}\Big(\int_{y\in{(Q_0^*)}^c}{\dfrac{1}{|y-c_{Q_0}|^{N-{L}/{p}}}}dy\Big)^p \\
&\lesssim 2^{-kp(N-2d+{d}/{p}-L/p)}2^{\mu dp({1}/{p}-1)}2^{\mu p(N-{L}/{p}-d)}
\end{align*} for ${L}>d(1-p)$ and $N-{L}/{p}>d$. Here the third inequality follows from the fact that 
\begin{equation*}
|y-x_Q|\gtrsim |y-c_{Q_0}|
\end{equation*} for $y\in (Q_0^*)^c$ and $Q\subset Q_0$, and the fourth one holds because of (\ref{rcondition}). 
Finally we obtain
\begin{equation*}
(\ref{s2}) \lesssim  1.
\end{equation*}

For (\ref{s1}) let $K_k(x,y)$ be the kernel of $T_{b_k}$ and write
\begin{equation*}
\big\Vert T_{b_k}(\chi_{Q_{0}^*}R_{Q_0,k})\Vert_{L^p(\widetilde{Q_0^{\rho}}^c)} \leq \Big[ \int_{\widetilde{Q_0^{\rho}}^c}{ \Big(  \int_{Q_0^*}{\big|  K(x,y) \big||R_{Q_0,k}(y)|}dy  \Big)^p }dx \Big]^{1/p}.
\end{equation*} 
Using H\"older's inequality and (\ref{rcondition})
it is less than
\begin{align*}
&2^{\mu\rho( L/p-d/p+d  )}\Vert  R_{Q_0,k}\Vert_{L^1}\sup_{y\in Q_0^*}{\int_{\widetilde{Q_0^{\rho}}^c }{|x-y|^{L/p}|K_k(x,y)|}dx}\\
&\leq 2^{\mu\rho L/p}2^{\mu d(1-\rho)(1/p-1)}\sup_{y\in Q_0^*}{\int_{\widetilde{Q_0^{\rho}}^c }{|x-y|^{L/p}|K_k(x,y)|}dx}
\end{align*} for $L>d(1-p)$. Here we select $L>d(1-p)$ such that $L/p$  becomes an integer.
Then it suffices to show  \begin{equation}\label{ooo}
\sup_{y\in Q_0^*}{\int_{\widetilde{Q_0^{\rho}}^c }{|x-y|^{L/p}\big|K_k(x,y)\big|}dx}
\lesssim 2^{-k\rho L/p} 2^{-kd(1-\rho)(1/p-1)}2^{\epsilon (\mu-k)}
\end{equation} for some $\epsilon>0$.
Let $y\in Q_0^*$. By Cauchy-Schwarz inequality and the fact that $|x-c_{Q_0}|\lesssim |x-y|$ one obtains
\begin{equation*}
\int_{\big(  \widetilde{Q_0^{\rho}} \big)^{c}}{ |x-y|^{L/p}\big| K_k(x,y) \big|  }dx\lesssim2^{\mu \rho(|\alpha|-d/2)}\Big( \int_{\mathbb{R}^d}{\big| (x-y)^{\alpha+\beta}K_k(x,y)  \big|^2}dx \Big)^{1/2}
\end{equation*} for a multi-index $\beta$ with $|\beta|=L/p$. 
Define 
\begin{equation*}
c_k(y,\eta):=\int_{\mathbb{R}^d}{K_k(x+y,y)e^{-2\pi i \langle x,\eta \rangle }}dx.
\end{equation*} Then by Plancherel's theorem
\begin{equation*}
\Big( \int_{\mathbb{R}^d}{\big| (x-y)^{\alpha+\beta}K_k(x,y)  \big|^2}dx \Big)^{1/2}=\Big( \int_{\mathbb{R}^d}{\big| \partial_{\eta}^{\alpha+\beta}c_k(y,\eta)  \big|^2}d\eta \Big)^{1/2}.
\end{equation*}
Observe that $\overline{c_k(y,\eta)}$ can be interpreted as a symbol corresponding to the adjoint operator of $T_{{b_k}}$ and therefore $c_k$ belongs to $\mathcal{S}_{\rho,\rho}^m$(See Appendix for more details). Furthermore $\eta$ lives in  the annulus $\{\eta: 2^{k-2}\leq |\eta|\leq 2^{k+2}\}$. Therefore we have
\begin{equation*}
\int_{\big(  \widetilde{Q_0^{\rho}} \big)^{c}}{ |x-y|^{L/p}\big| K_k(x,y) \big|  }dx
\lesssim 2^{-k\rho L/p}2^{-kd(1-\rho)(1/p-1)}2^{\rho (|\alpha|-d/2)(\mu-k)},
\end{equation*} which concludes (\ref{ooo}).\\

Now assume $l(Q_0)> 2^{-3} $( i.e. $\mu<3$ ). In this case we employ the range $k\geq 3$ in our summations since $\mu\leq 2$. 
Then by repeating the above process we see that 
\begin{equation*}
\Big( \sum_{k=3}^{\infty}{\big\Vert T_{b_k}R_{Q_0,k}\Vert_{L^p({Q_0^{**}})}^p} \Big)^{{1}/{p}} \lesssim 1,
\end{equation*}
\begin{equation*}
\Big( \sum_{k=3}^{\infty}{\big\Vert T_{b_k}(\chi_{(Q_0^*)^c}R_{Q_0,k})\Vert_{L^p(({Q_0^{**}})^c)}^p} \Big)^{{1}/{p}} \lesssim 1,
\end{equation*} and 
\begin{equation*}
\Big( \sum_{k=3}^{\infty}{\big\Vert T_{b_k}(\chi_{Q_{0}^*}R_{Q_0,k})\Vert_{L^p( ({Q_0^{**}})^c)}^p} \Big)^{{1}/{p}} \lesssim 1
\end{equation*} 
uniformly in $\mu\leq 2$.
This completes the proof of the case $0<p\leq 1$ in Theorem \ref{main}.

\subsection{ Proof of (\ref{goal}); the case $2<p<\infty$  }

Suppose $2<p<\infty$ and we will prove (\ref{goal}) with $0<t\leq \infty $ and $q=p$.
In \cite{Pr_Ro_Se} the boundedness of (\ref{multiplierexample}) from $F_p^{0,p}$ into $F_{p}^{0,t}$ was established as a corollary of the following result.
Let $0<a<d$, $\epsilon>0$ and $1<p_0<p<\infty$. Consider operators $T_k$ defined on  $\mathcal{S}(\mathbb{R}^d)$ by 
\begin{equation*}
T_kf(x):=\int_{\mathbb{R}^d}{K_k(x,y)f(y)}dy,
\end{equation*} where each $K_k$ is a continuous and bounded kernel. 
Assume that $T_k$ satisfies 
\begin{equation}\label{con1}
\sup_{k>0}{2^{ka/{p}}\Vert  T_k \Vert_{L^{p}\to L^{p}}}\leq A
\end{equation} and
\begin{equation}\label{con2}
\sup_{k>0}{2^{ka/{p_0}}\Vert  T_k \Vert_{L^{p_0}\to L^{p_0}}}\leq B_0.
\end{equation}
Furthermore let $\Gamma\geq 1$ and assume that for each cube $Q$ there is a measurable set $\mathcal{E}_Q$ so that 
\begin{equation*}
|\mathcal{E}_Q|\leq \Gamma \max\{|Q|^{1-a/d},|Q|\},
\end{equation*} and for every $k\in\mathbb{N}$ and every cube $Q$ with $2^kl(Q)\geq 1$,
\begin{equation}\label{con3}
\sup_{x\in Q}{\int_{\mathbb{R}^d\setminus \mathcal{E}_Q}{|K_k(x,y)|}dy } \leq B_1 \max\big\{ (2^kl(Q))^{-\epsilon},2^{-k\epsilon}   \big\}.
\end{equation}
Let $\mathcal{B}=B_0^{p_0/p}(A\Gamma^{1/p}+B_1)^{1-p_0/p}$. Then there exists a $C>0$ such that
\begin{equation*}
\Big\Vert  \Big(  \sum_{k}{2^{kar/p}|\phi_k\ast T_kf_k|^q}   \Big)^{1/q}   \Big\Vert_{L^p} \leq CA\Big[ \log(3+\dfrac{\mathcal{B}}{A})  \Big]^{1/q-1/p}\Big(\sum_{k}{\Vert  f_k \Vert_{L^p}^p} \Big)^{1/p}.
\end{equation*}

We set $a=(1-\rho)d$ and  for cube Q with $l(Q)<1$ choose $\mathcal{E}_Q$ to be the cube with the same center, but  with diameter $Cl(Q)^{\rho}$  for large $C$. If $l(Q)\geq 1$, $\mathcal{E}_Q$  is  just a dilate of Q by a factor of large constant $C$. Let 
\begin{equation}
\widetilde{b}_k(x,\xi):=(1+|\xi|^2)^{-d(1-\rho)/(2p)}b_k(x,\xi)
\end{equation} and define $T_k:=T_{\widetilde{b}_k}$.
Then operators $T_k$ obviously satisfy (\ref{con1}) and (\ref{con2}) with $p_0=2$. Thus it suffices to show that (\ref{con3}) still holds with our kernel
\begin{equation}
K_k(x,y)=\int_{\mathbb{R}^d}{\widetilde{b}_k(x,\xi)e^{2\pi i\langle x-y,\xi\rangle }}d\xi.
\end{equation}
Fix $x\in Q$ and then for $y\in \mathbb{R}^d\setminus \mathcal{E}_Q$ we see $|x-y|\gtrsim l(\mathcal{E}_Q)$. 
Therefore
\begin{align*}
\int_{\mathbb{R}^d\setminus \mathcal{E}_Q}{|K_k(x,y)|}dy&\leq \int_{|x-y|\gtrsim l(\mathcal{E}_Q)}{|K_k(x,y)|}dy\\
    &\lesssim l(\mathcal{E}_Q)^{-|\alpha|+d/2}\Big( \int_{\mathbb{R}^d}{|(x-y)^{\alpha}K_k(x,y)|^2}dy  \Big)^{1/2}\\
    &= l(\mathcal{E}_Q)^{-|\alpha|+d/2}\Big( \int_{\mathbb{R}^d}{|\partial_{\xi}^{\alpha}\widetilde{b_k}(x,\xi)|^2}d\xi  \Big)^{1/2}
\end{align*} for any multi-indices $\alpha$ with $|\alpha|>d/2$ by Cauchy-Schwarz inequality and Plancherel's theorem.
Since $\widetilde{b}_k\in S_{\rho,\rho}^{-d(1-\rho)/2}$ this is bounded by
\begin{equation*}
\big(2^{k\rho}l(\mathcal{E}_Q)\big)^{-|\alpha|+d/2}.
\end{equation*}
By choosing $\alpha$ satisfying $\rho(|\alpha|-d/2)>\epsilon$ we prove (\ref{con3}) and it completes the proof of (\ref{goal}).

\section{\textbf{Proof of Theorem \ref{besov}}}\label{bproof}

Theorem \ref{besov} can be proved in a similar way. The first one is simply from (\ref{finalcoral}), H\"older's inequality, and the embedding theorem $l^{p_1}\subset l^{p_2}$ for $p_1\leq p_2$. By repeating the process in Section \ref{basic}, (\ref{12}) and (\ref{34}) hold if  $F$-spaces are replaced by $B$-spaces, and the boundedness of $T^{(3)}$ on $B$-spaces follows just from (\ref{pqr}).

\section{\textbf{Proof of Theorem \ref{mmm} and \ref{mmmm}}}\label{infinityproof}

Now we consider the case $p=\infty$. Unlike the case $1<p<\infty$ we do not have  $L^{\infty}$ boundedness of the operator. Instead, Fefferman \cite{Fe} proved that $T_a$ maps $L^{\infty}$ into $BMO$ and the key idea of the proof is 
the following $L^{\infty}$ estimates with an additional support condition of $a$.
For any $a\in \mathcal{S}_{\rho,\delta}^{m}$ and $r>0$ if $a(x,\cdot)$ is supported  in $\{\xi:r/2 \leq |\xi|\leq 2r\}$ and $m=-d(1-\rho)/2$ then one has 
\begin{equation}\label{flemma}
\big\Vert T_{a}f \big\Vert_{L^{\infty}}\lesssim \Vert a \Vert_{\mathcal{S}_{\rho,\delta}^{m}}\Vert f\Vert_{L^{\infty}}
\end{equation} 
 where 
 \begin{equation*}
 \Vert  a \Vert_{\mathcal{S}_{\rho,\delta}^{m}}=\sup_{|\alpha|,|\beta|\leq d}{\big|  \partial_{\xi}^{\alpha}\partial_x^{\beta}a(x,\xi) \big|(1+|\xi|)^{-m+\rho|\alpha|-\delta|\beta|}}.
 \end{equation*}
 Here the implicit constant is independent of $r$.
Also if $a(x,\cdot)$ is supported in a ball of radius $R$ centered at the origin for some constant $R>0$, then there exists $C_R>0$ such that
\begin{equation}\label{flemma2}
\big\Vert  T_{a}f\big\Vert_{L^{\infty}}\leq C_{R}\Vert a \Vert_{\mathcal{S}_{\rho,\delta}^{m}}\Vert f\Vert_{L^{\infty}}.
\end{equation}

\subsection{Proof of Theorem \ref{mmm}}
One may assume $s_1=s_2=0$ without loss of generality. 
We will prove that for $0<t<1$
\begin{equation*}
\big\Vert T_af\big\Vert_{F_{\infty}^{0,t}}\lesssim \Vert f\Vert_{F_{\infty}^{0,\infty}}
\end{equation*} and
other cases follow by embedding 
$F_{\infty}^{0,q_1}\hookrightarrow F_{\infty}^{0,q_2}$ for $0<q_1<q_2\leq \infty$. Let $m=-d(1-\rho)/2$ and $f\in F_{\infty}^{0,\infty}$. 
Note that \begin{equation*}
\Vert T_af\Vert_{F_{\infty}^{0,t}} \leq \sum_{k=0}^{10}{\big\Vert \phi_k\ast T_af \big\Vert_{L^{\infty}}}+\sup_{l(Q)<1}{\Big(\frac{1}{|Q|}\int_Q{\sum_{k\geq 10-\log_2{l(Q)}}{  |\phi_k\ast T_af(x)|^t}}dx\Big)^{1/t}}.
\end{equation*}
We consider just the supremum term and a similar method can be applied to the first one.
We shall base the proof on the arguments in Section \ref{basic}  and use same notations.
Write 
\begin{equation*}
T_af=T^{(1)}f+T^{(2)}f+T^{(3)}f
\end{equation*} and it suffices to show that for a fixed dyadic cube $Q$ of side length $l(Q)<1$
\begin{equation}\label{fffgoal}
 \frac{1}{|Q|}\int_Q{\sum_{k\geq 10-\log_2{l(Q)}}{\big|\phi_k\ast T^{(i)}f(x)\big|^t}}dx   \lesssim \sup_{k}{\Vert  f_k \Vert^t_{L^{\infty}}}
\end{equation} uniformly in $Q$ for each $i=1,2,3$, 
where $\widetilde{\phi}_k=\phi_{k-1}+\phi_k+\phi_{k+1}$ and $f_k=\widetilde{\phi}_k\ast f$.

First of all, one has
\begin{equation*}
 \frac{1}{|Q|}\int_Q{\sum_{k\geq 10-\log_2{l(Q)}}{\big|\phi_k\ast T^{(1)}f(x)\big|^t}}dx  \lesssim \sum_{k\geq 10-\log_2{l(Q)}}{  \sum_{j={k-2}}^{k+2}{\sum_{n=0}^{j-3}{  \Vert T_{a_{j,n}}f_n \Vert^t_{L^{\infty}}    }}    }.
\end{equation*} 
By using the same argument to get (\ref{firstkey}) with $r=1$ and $c^{(r,d)}=d+1$, we obtain
\begin{align}\label{firstkey1}
\Vert  T_{a_{j,n}} f_n  \Vert_{L^{\infty}}&\lesssim  2^{k(m+(d+1)(1-\rho)+\rho N)-jN}\big\Vert \mathcal{M} f_n \big\Vert_{L^{\infty}}\nonumber
   &\leq  2^{k(m+(d+1)(1-\rho)+\rho N)-jN}  \sup_{l}{\Vert f_l \Vert_{L^{\infty}}}
   \end{align} for sufficiently large $N$. This proves (\ref{fffgoal}) when $i=1$.

Now we deal with $T^{(2)}$. We break up this operator into two parts as (\ref{secondoperator}). Then
 (\ref{flemma}) and  (\ref{flemma2}) yield the desired result for the finite sum. For the infinite sum, we see that
\begin{equation*}
 \frac{1}{|Q|}\int_Q{\sum_{k\geq 10-\log_2{l(Q)}}{  \Big|  \phi_k\ast\Big(\sum_{n=10}^{\infty}{T_{a_n} f}\Big)(x)   \Big|^t }}dx \lesssim \sum_{k\geq 10-\log_2{l(Q)}}{\sum_{n=10}^{\infty}{ \big\Vert  \phi_k\ast (T_{a_n} f_n)   \big\Vert^t_{L^{\infty}}    }}.
\end{equation*}
Recall that $a_n(x,\xi)=\Phi_n\ast a(\cdot,\xi)(x)\widehat{\phi_n}(\xi)$ where $\Phi_n=\sum_{j=n-2}^{n+2}{\phi_j}$.
From the support properties of $\widehat{\phi_k}$ and $\widehat{T_{a_n}f_n}$  it follows immediately that the summand vanishes unless $k\leq n+4$. Thus the last expression is bounded by a constant times
\begin{equation*}
\sum_{k\geq 10-\log_2{l(Q)}}{\sum_{n={k-4}}^{\infty}{ \big\Vert   T_{a_n} f_n   \big\Vert^t_{L^{\infty}}    }}.
\end{equation*} 
Now our claim is that for any $J>0$ \begin{equation}\label{compactcon}
\Vert T_{a_n}f_n \Vert_{L^{\infty}}\lesssim_J 2^{-Jn} \sup_{l}{\Vert f_l\Vert_{L^{\infty}}},
\end{equation} which completes the proof for $T^{(2)}$.
To see (\ref{compactcon}) we apply the size estimate (\ref{sizeest}), but it may not be true when we drop the hypothesis of compact support of $a(x,\xi)$ in $x$ variable.  Thus, first define $a^{\tau}(x,\xi)$ as (\ref{supvar}) and let $a^{\tau}_n(x,\xi):=\Phi_n\ast a^{\tau}(\cdot,\xi)(x)\widehat{\phi_n}(\xi)$. When $K_n^{\tau}(x,y)$ is the kernel of $T_{a_n^{\tau}}$, then  for any $J>0$ and $N>0$ we have 
\begin{equation*}
|K_n^{\tau}(x,y)|\lesssim_{J,N}2^{-Jn}\frac{1}{(1+|x-y|)^N}
\end{equation*}  uniformly in $\tau$. Thus, 
\begin{equation*}
\Vert T_{a^{\tau}_n}f_n \Vert_{L^{\infty}}\leq\big\Vert   f_n \big\Vert_{L^{\infty}} \sup_{x\in\mathbb{R}^d}{\int_{\mathbb{R}^d}{|K_n^{\tau}(x,y)|}dy}\lesssim_J2^{-Jn}\sup_{l}{\Vert  f_l \Vert_{L^{\infty}}}
\end{equation*} and this estimate holds uniformly in $\tau$.
Then (\ref{compactcon}) follows from the fact that
\begin{align*}
\limsup_{\tau\to\infty}{\big\Vert T_{a_n^{\tau}}f_n-T_{a_n}f_n  \big\Vert_{L^{\infty}}}= \limsup_{\tau\to\infty}{\sup_{|x|\geq\tau}{|T_{a_n}f_n(x)|}} \lesssim_n \limsup_{\tau\to\infty}{\tau^{-1}}=0
\end{align*} where the inequality follows by an integration by parts.

For the operator $T^{(3)}$, as in Section \ref{basic}, we write $T^{(3)}f$  as 
\begin{equation*}
T^{(3)}f=\sum_{n=3}^{\infty}{T_{b_n}f}
\end{equation*} where
$b_n(x,\xi)=\sum_{j=0}^{n-3}{a_{j,n}(x,\xi)}=\big( \sum_{j=0}^{n-3}{\phi_j} \big)\ast a(\cdot,\xi)(x)\widehat{\phi_n}(\xi)$.
Since $\widehat{T_{b_n}f}$ is supported in $\{\xi:2^{n-2}\leq|\xi|\leq 2^{n+2}\}$ we have
\begin{align*}
& \frac{1}{|Q|}\int_Q{\sum_{k\geq 10-\log_2{l(Q)}}{\big|\phi_k\ast T^{(3)}f(x)\big|^t}}dx\\
&\lesssim \sum_{k\geq 10-\log_2{l(Q)}}{ \sum_{n={k-2}}^{k+2}{  \frac{1}{|Q|}\int_Q{\big| \phi_k\ast (T_{b_n}f_n)(x) \big|^t}dx   }  }.
\end{align*} We consider only the case $n=k$ and other cases follow by the same way.
Now let us apply Fefferman's method in \cite{Fe}. Let $\Psi$ be a bump function satisfying $0\leq\Psi\leq 10$, $\Psi\geq 1$ on $[-1/2,1/2]^d$, and $Supp(\widehat{\Psi})\subset \{\xi: |\xi|\leq 2^{-3}\}$. Define $\Psi_Q(x):=\Psi\big(l(Q)^{-\rho}(x-c_Q) \big)$ where $c_Q$ is the center of cube $Q$. Then $\Psi_Q$ has the properties
$0\leq \Psi_Q\leq 10$,
$\Psi_Q\geq 1$ on  $Q$,
$Supp(\widehat{\Psi_Q})\subset \{\xi:|\xi|\leq 2^{-3}l(Q)^{-\rho}\}$,
and
$\Vert \widehat{\Psi_Q}\Vert_{L^{\infty}}\lesssim l(Q)^{\rho d}$.
Then it follows that
\begin{align*}
& \sum_{k\geq 10-\log_2{l(Q)}}\frac{1}{|Q|}\int_Q{{\big|\phi_k\ast(T_{b_k} f_k)(x)\big|^t}}dx\nonumber\\
&\leq \sum_{k\geq 10-\log_2{l(Q)}}\frac{1}{|Q|}\int_Q{{\big|\phi_k\ast(T_{b_k} f_k)(x)\Psi_Q(x)\big|^t}}dx\nonumber\\
&\lesssim \sum_{k\geq 10-\log_2{l(Q)}}{\frac{1}{|Q|}\int_Q{    \big| \phi_k\ast (T_{b_k}f_k)(x)\Psi_Q(x)-\phi_k\ast \big(\Psi_QT_{b_k}f_k\big)(x)  \big|^t   }dx  }\\
&\relphantom{=} +\sum_{k\geq 10-\log_2{l(Q)}}{\frac{1}{|Q|}\int_Q{  \big| \phi_k\ast \big(\Psi_Q T_{b_k}f_k-T_{b_k}(f_k \Psi_Q)\big)(x)  \big|^t     }dx  }\\
&\relphantom{=} +\sum_{k\geq 10-\log_2{l(Q)}}\frac{1}{|Q|}\int_Q{{\big|\phi_k\ast\big( T_{b_k}(f_k \Psi_Q)\big)(x)\big|^t}}dx \\
&:= I+II+III.
\end{align*}

By elementary computation one obtains 
\begin{equation*}
\big\Vert \phi_k\ast(T_{b_k}f_k)\Psi_Q-\phi_k\ast\big(\Psi_QT_{b_k}f_k\big)  \big\Vert_{^{\infty}}\lesssim 2^{-k}l(Q)^{-\rho}\Vert T_{b_k}f_k \Vert_{L^{\infty}}.
\end{equation*}
Furthermore, since $\Vert b_k\Vert_{\mathcal{S}_{\rho,\rho}^m}\lesssim \Vert a\Vert_{\mathcal{S}_{\rho,\rho}^{m}}$ uniformly in $k$ it follows that
\begin{equation}\label{secondkey}
\Vert T_{b_k}f_k \Vert_{L^{\infty}}\lesssim \Vert a\Vert_{\mathcal{S}_{\rho,\rho}^{m}}\sup_{l}{\Vert f_l\Vert_{L^{\infty}}}
\end{equation} by applying (\ref{flemma}). Combining these two estimates and summing over $k\geq 10-\log_2{l(Q)}$ one obtains
\begin{equation*}
I\lesssim \sup_{l}{\Vert f_l\Vert^t_{L^{\infty}}}.
\end{equation*}

For the second one,
$\Psi_QT_{b_k}f_k-T_{b_k}(f_k\Psi_Q)$ can be written in the form $T_{v_k}f_k$
where
\begin{equation*}
v_k(x,\xi)=\int_{\mathbb{R}^d}{ \big( b_k(x,\xi)-b_k(x,\eta+\xi)   \big)\widehat{\Psi_Q}(\eta)e^{2\pi i\langle x,\eta\rangle}         }d\eta
\end{equation*}  is a symbol in $\mathcal{S}_{\rho,\rho}^{m}$. 
Indeed, 
\begin{equation}\label{symbolnorm}
\Vert v_k\Vert_{\mathcal{S}_{\rho,\rho}^{m}}\lesssim 2^{-\rho k}l(Q)^{-\rho}\Vert a\Vert_{\mathcal{S}_{\rho,\rho}^{m}}.
\end{equation}

Combining (\ref{flemma}) and (\ref{symbolnorm}), 
\begin{equation*}
\big\Vert  \Psi_QT_{b_k}f_k-T_{b_k}(f_k\Psi_Q) \big\Vert_{L^{\infty}} \lesssim 2^{-\rho k}l(Q)^{-\rho}\sup_{l}{\Vert f_l\Vert_{L^{\infty}}},
\end{equation*} which establishes
\begin{equation*}
II  \lesssim \sup_{l}{\Vert f_l\Vert^t_{L^{\infty}}}.
\end{equation*}

For the last one we apply H\"older's inequality with $2/t>1$, Young's inequality, and (\ref{pqr}), and then 
\begin{align*}
III&\lesssim \sum_{k\geq 10-\log_2{l(Q)}}{ \frac{1}{|Q|^{t/2}}\big\Vert  T_{b_k}(f_k\Psi_Q) \big\Vert^t_{L^2}   }\lesssim\frac{1}{|Q|^{t/2}}\sum_{k\geq 10-\log_2{l(Q)}}{  2^{-kdt(1-\rho)/2}\Vert  f_k\Psi_Q \Vert_{L^2}^t }\\
     &\leq \frac{1}{|Q|^{t/2}}\sup_{l}{\Vert  f_l\Vert^t_{L^{\infty}}}\sum_{k\geq 10-\log_2{l(Q)}}{  2^{-kdt(1-\rho)/2}\Vert \Psi_Q \Vert_{L^2}^t }\lesssim\sup_{l}{\Vert f_l\Vert^t_{L^{\infty}}}.
\end{align*}
This ends the proof for $T^{(3)}$.

\subsection{Proof of Theorem \ref{mmmm}}
By (\ref{firstkey}) and (\ref{secondkey}), $T^{(1)}$ and $T^{(2)}$ map $B_{\infty}^{s_1,q}$ into $B_{\infty}^{s_2,t}$ for any $0<q,t\leq\infty$. The boundedness of $T^{(3)}$ is immediately from (\ref{flemma}) for both cases (1) and (2).

\section{\textbf{Proof of Theorem \ref{sharptheorem1} and \ref{sharptheorem2}}}\label{sharpresult}
 We may assume $s_1=s_2=0$.

\subsection{Proof of Theorem \ref{sharptheorem1} ; The case $0<p<\infty$}\label{ghghghgh}

\subsubsection{Proof of Theorem \ref{sharptheorem1} (1)}
Recall that for $0<p<\infty$  $h^p$ boundedness of $c_{m,\rho}(D)$ does not hold without the assumption (\ref{cexample}). For details
see \cite{Fe, Hi, Wa} for $1<p<\infty$ and \cite{Mi} for $0<p\leq 1$.
Suppose $m>-d(1-\rho)\big|1/2-1/p \big|$ and
choose $\epsilon>0$ such that 
\begin{equation*}
m-2\epsilon >-d(1-\rho)\big| {1}/{2}-{1}/{p}  \big|.
\end{equation*}
Then we know that $c_{m-2\epsilon,\rho}(D)$ is not bounded in $F_p^{0,2} (=h^p)$ and thus there exists $f\in F_{p}^{0,2}$ so that $\big\Vert c_{m-2\epsilon,\rho}(D)f\big\Vert_{F_p^{0,2}}=\infty$. We define 
\begin{equation*}
g(x):=\sum_{k=0}^{\infty}{2^{-\epsilon k}\phi_k\ast f(x)}
\end{equation*} and then observe that
\begin{equation*}
\Vert g\Vert_{F_p^{\epsilon,2}}\approx\Vert f\Vert_{F_p^{0,2}}<\infty
\end{equation*} and 
\begin{equation*}
\big\Vert c_{m,\rho}(D)g\big\Vert_{F_p^{-\epsilon,2}}\approx\big\Vert c_{m-\epsilon,\rho}(D)g\big\Vert_{F_p^{0,2}}\approx\big\Vert c_{m-2\epsilon,\rho}(D)f \big\Vert_{F_p^{0,2}}=\infty.
\end{equation*}
Then the embeddings $F_p^{\epsilon,2}\hookrightarrow F_p^{0,q}$ and $F_p^{0,t}\hookrightarrow F_p^{-\epsilon,2}$ proves (1).

\subsubsection{Proof of Theorem \ref{sharptheorem1} (2) }

Now suppose  
\begin{equation*}
m=-d(1-\rho)\big|{1}/{2}-{1}/{p} \big|
\end{equation*} and consider the condition of $q$ and $t$. 
Christ and Seeger \cite{Ch_Se}  show that $C_m,\rho$(D) in (\ref{multiplierexample}) is unbounded in $F_p^{0,q}$ provided that $0<q<p\leq 2$ with 
\begin{equation*}
m=-d(1-\rho)\big({1}/{p}-{1}/{2}\big).
\end{equation*}
For $r>0$ we define  $\mathcal{E}(r)$ to be the space of all distributions whose Fourier transforms are supported in $\{\xi: |\xi|\leq 2r\}$  as before.
\begin{theorem} \cite{Ch_Se}\label{Se}
Let $0<q<p\leq 2$, $0<\rho < 1$, $m=-d(1-\rho)\big({1}/{p}-{1}/{2}\big)$. Then, for $R\geq 2$, 
\begin{equation*}
\sup{\{\big\Vert  c_{m,\rho}(D)f   \big\Vert_{F_{p}^{0,q}}:\Vert   f \Vert_{F_{p}^{0,q}}\leq 1,f\in\mathcal{E}(R)\}} \approx (\log{R})^{{1}/{q}-{1}/{p}}.
\end{equation*} 
\end{theorem}
The case $"\lesssim"$ is immediate by H\"older's inequality and the embedding $l^q\subset l^p$.
They use a randomization technique to show the existence of $f\in \mathcal{E}(R)$ such that 
\begin{equation*}
\Vert f\Vert_{F_p^{0,q}}\lesssim 1
\end{equation*} and 
\begin{equation*}
\Vert c_{m,\rho}(D)f \Vert_{F_p^{0,q}}\gtrsim (\log{R})^{1/q-1/p},
\end{equation*}
and this idea can be applied to our cases.
For each $k\in \mathbb{Z}_+=\{1,2,\dots\}$ let $\mathcal{Q}(k)$ be the set of all dyadic cubes of side length $2^{-k}$ in $[0,1]^d$ and $\mathcal{Q}=\bigcup_{k\in \mathbb{Z}_+}{\mathcal{Q}(k)}$. Let $\Omega$ be a probability space with probability measure $\mu$. Let $\{\theta_{Q}\}$ be a family of independent random variables indexed by $Q\in\mathcal{Q}$, each of which takes the value $1$ with probability $2^{-kd(1-\rho)}$ and the value $0$ with probability $1-2^{-kd(1-\rho)}$ for $Q\in\mathcal{Q}(k)$. Let $\eta$ be in $\mathcal{E}(1)$ such that $\widehat{\eta}$ vanishes identically in a neighborhood of the origin and $\widehat{\eta}(\xi)=1$ if $2^{-1/2}\leq |\xi|\leq 2^{1/2}$ and let $\widetilde{\eta}$ be in $\mathcal{E}(1)$ whose Fourier transform equals $1$ on the support of $\widehat{\eta}$.
Define 
for $k\geq 1$ the operator $S_k$ by 
\begin{equation}\label{operator}
\widehat{S_kf}(\xi):=2^{mk}e^{2\pi i|\xi|^{1-\rho}}\widehat{\eta}(2^{-k}\xi)\widehat{f}(\xi)
\end{equation} and for $w\in\Omega$
\begin{equation*}
f^{k,w}(x):=2^{kd(1-\rho)/p}\sum_{Q\in\mathcal{Q}(k)}{\theta_{Q}(w)\widetilde{\eta}(2^k(x-c_Q))}.
\end{equation*}
According to \cite{Ch_Se}, 
\begin{equation*}
\Big( \int_{\Omega}{\Big\Vert  \Big( \sum_{k=1}^{L}{|f^{k,w}|^q}  \Big)^{1/q}   \Big\Vert_{L^p}^p}d\mu(w)  \Big)^{1/p} \lesssim L^{1/p}
\end{equation*} for any $0<q\leq \infty$, and 
\begin{equation*}
\Big( \int_{\Omega}{\Big\Vert  \big( \sum_{k=1}^{L}{|S_kf^{k,w}|^t}  \big)^{1/t}   \Big\Vert_{L^p}^p}d\mu(w)  \Big)^{1/p} \gtrsim L^{1/t}
\end{equation*} for $0<t<p\leq 2$.
This implies that 
\begin{equation}\label{newco}
\sup{\{\big\Vert  C_{m,\rho}(D)f   \big\Vert_{F_{p}^{0,t}}:\Vert   f \Vert_{F_{p}^{0,q}}\leq 1,f\in\mathcal{E}(R)\}} \gtrsim (\log{R})^{{1}/{t}-{1}/{p}}.
\end{equation} for $0<t<p$ and $0<q\leq \infty$, which proves (2).

\subsubsection{Proof of Theorem \ref{sharptheorem1} (3) }
Suppose $2\leq p<\infty$ and $p<q$.
Let $p<r<\infty$.
Note that the adjoint operator $\big(c_{m,\rho}(D)\big)^*=\overline{c_{m,\rho}}(D)$ has the same estimate like (\ref{newco}).
Thus for sufficiently large $R>0$ there exists $g\in\mathcal{E}(R)$ such that 
\begin{equation*}
\Vert g\Vert_{F_{p'}^{0,1}}\leq1
\end{equation*} and 
\begin{equation*}
\big\Vert \overline{c_{m,\rho}}(D)g \big\Vert_{F_{p'}^{0,r'}}\gtrsim (\log{R})^{1/r'-1/p'}.
\end{equation*}
Since for some constant $A>0$ 
\begin{align*}
\big\Vert \overline{c_{m,\rho}}(D)g\big\Vert_{F_{p'}^{0,r'}}&= \big\Vert    \big\{ \overline{c_{m,\rho}}(D)(g\ast \phi_k) \big\}\big\Vert_{L^{p'}(l^{r'})}\\
      &= \sup_{\Vert \{f_k\} \Vert_{L^p(l^r)}\leq 1}{\Big| \int_{\mathbb{R}^d}{\sum_{k=0}^{A\log{R}}{ \overline{c_{m,\rho}}(D)(g\ast \phi_k)(x) f_k(x)}}dx  \Big|}\\
      &= \sup_{\Vert \{f_k\} \Vert_{L^p(l^r)}\leq 1}{\Big| \int_{\mathbb{R}^d}{\sum_{k=0}^{A\log{R}}{ g\ast \phi_k(x) c_{m,\rho}(D)(\widetilde{\phi}_k\ast f_k)(x)}}dx  \Big|}\\
      &\leq  \sup_{\Vert \{f_k\} \Vert_{L^p(l^r)}\leq 1}{\Big\Vert  \sup_{0\leq k\leq A\log{R}}{\big| c_{m,\rho}(D)(\widetilde{\phi}_k\ast f_k) \big|}  \Big\Vert_{L^p}},
\end{align*}  there exists a sequence of functions $\{f_k\}$ in $L^p(l^r)$ such that
\begin{equation}\label{9}
 \Vert  \{f_k\} \Vert_{L^p(l^r)}\leq 1
\end{equation} and
\begin{equation}\label{10}
\Big\Vert  \sup_{0\leq k\leq  A\log{R}}{\big| c_{m,\rho}(D)(\widetilde{\phi}_k\ast f_k) \big|}  \Big\Vert_{L^p}\gtrsim (\log{R})^{1/p-1/r}.\end{equation}
Define 
\begin{equation*}
f(x):=\sum_{k=0}^{A\log{R}}{\widetilde{\phi}_k\ast f_k(x)}
\end{equation*} and  then clearly $f\in S'$ and 
\begin{equation*}
\Vert f\Vert_{F_p^{0,\infty}}\leq \Vert f\Vert_{F_p^{0,r}}\lesssim  1  
\end{equation*} by (\ref{newdual}) and (\ref{9}). Moreover, since
$\Vert c_{m,\rho}(D)f \Vert_{F_{p}^{0,\infty}}$ is comparable to the left hand side of (\ref{10}),
we see that \begin{equation*}
\big\Vert c_{m,\rho}(D)f\big\Vert_{F_p^{0,\infty}}\gtrsim (\log{R})^{1/p-1/r}.
\end{equation*}
We conclude that for $2\leq p <q\leq \infty$ and $0<t\leq \infty$
\begin{equation}\label{newcoo}
\sup{\{\big\Vert  c_{m,\rho}(D)f   \big\Vert_{F_{p}^{0,t}}:\Vert   f \Vert_{F_{p}^{0,q}}\leq 1,f\in\mathcal{E}(R^A)\}} \gtrsim_{\epsilon} (\log{R})^{\epsilon}
\end{equation} with $0<\epsilon<1/p-1/q$ ( In fact, if $p<q<\infty$ then we can put $q=r$ and $\epsilon=1/p-1/q$ ).
This completes the proof of (3).

\subsection{Proof of Theorem \ref{sharptheorem2}; The case $0<p<\infty$}
\subsubsection{Proof of Theorem \ref{sharptheorem2} (1)}
We follow the same idea in the proof of Theorem \ref{sharptheorem1} (1) and apply $B_p^{0,t}\hookrightarrow F_p^{-\epsilon,p}$ and $F_p^{\epsilon,p} \hookrightarrow B_p^{0,q}$, instead of $F_p^{\epsilon,2}\hookrightarrow F_p^{0,q}$ and $F_p^{0,t}\hookrightarrow F_p^{-\epsilon,2}$.

\subsubsection{Proof of Theorem \ref{sharptheorem2} (2)}
Assume $m=-d(1-\rho)\big| 1/p-1/2  \big|$ and $q>t$.
For sufficiently large $R>0$ we will construct $h\in\mathcal{E}(R)$ so that
\begin{equation*}
\Vert h\Vert_{B_p^{0,q}}\lesssim 1 \quad \text{uniformly in}~R,
\end{equation*} and
\begin{equation*}
\big\Vert c_{m,\rho}(D)h\big\Vert_{B_p^{0,t}}\gtrsim C_R
\end{equation*} where $C_R$ blows up to infinity as $R$ increases.

We first assume $0<p\leq 2$. 
Let $S_k$ be defined as (\ref{operator}) and  $h_k(x)=k^{-1/t}2^{kd/p}\widetilde{\eta}(2^kx)$.
Since $\Vert h_k \Vert_{L^p}\approx k^{-1/t}$ it is clear that
\begin{equation}\label{fest}
\Big(  \sum_{k=10}^{L}{\Vert h_k \Vert_{L^p}^q} \Big)^{1/q}  \lesssim 1 \quad \text{uniformly in } ~L.
\end{equation} 
Now our claim is
\begin{equation}\label{sest}
\Big(  \sum_{k=10}^{L}{\Vert S_kh_k \Vert_{L^p}^t} \Big)^{1/t} \gtrsim ( \log{L} )^{1/t}.
\end{equation} 
When $K_k$ is the convolution kernel of $S_k$,
 \begin{equation*}
 S_kh_k(x)= k^{-1/t}2^{-kd(1-1/p)}2^{km}K_k(x). 
 \end{equation*} because the Fourier transform of $\widetilde{\eta}$ is $1$ on the support of $\widehat{\eta}$. 
By the method of stationary phase as in \cite{Ch_Se}, for a  suitable $\epsilon_1>0$ there is the uniform estimate for large $k$
\begin{equation*}
|K_k(x)|\geq 2^{kd(1+\rho)/2} \quad \text{if}~(1-\epsilon_1)2^{-k\rho}\leq |x| \leq (1+\epsilon_1)2^{-k\rho}.
\end{equation*}  This gives
\begin{equation*}
\big\Vert S_kh_k \big\Vert_{L^p}\gtrsim k^{-1/t},
\end{equation*} which implies (\ref{sest}).

For $2\leq p <\infty$ we can choose a sequence of functions $\{f_k\}$ whose Fourier transform has a compact support in $\{\xi : |\xi|\approx 2^k\}$ such that \begin{equation*}
\Vert S_k(f_k)\Vert_{L^p}\gtrsim k^{-1/t} \gtrsim \Vert f_k\Vert_{L^p}
\end{equation*} by using the duality property  $(L^{p})^*=L^{p'}$. Then (\ref{fest}) and (\ref{sest}) hold with $f_k$ instead of $h_k$ for $q>t$.

\subsection{Proof of Theorem \ref{sharptheorem2}; The case $p=\infty$}

Let $e_1=(1,0,\dots,0)\in\mathbb{R}^d$.
For each $k=10,11,12,\dots$ we define a Schwartz function $g_k$ to satisfy \begin{equation*}
\widehat{g_k}(\xi)=\sigma_k 2^{-\rho kd}\widehat{\phi}(2^{2-\rho k}(\xi-2^ke_1))
\end{equation*} for some positive numbers $\sigma_k$ to be chosen later and define $S_k$ as in (\ref{operator}) again.
Then \begin{equation*}
\widehat{S_kg_k}(\xi)=\sigma_k 2^{mk}2^{-\rho kd}e^{2\pi i|\xi|^{1-\rho}}\widehat{\phi}(2^{2-\rho k}(\xi-x^{k})).
\end{equation*} Let $U_k$ be the Fourier transform of $e^{2\pi i|\cdot |^{1-\rho}}\widehat{\phi}(2^{2-\rho k}(\cdot-x^{k}))$ and then
a stationary phase calculation yields that for a suitable $\epsilon>0$ 
\begin{equation*}
\big| U_k(x)   \big|\gtrsim 2^{kd(1+\rho)/2} 
\end{equation*} if  $(1-\epsilon_0)2^{\rho k} \leq \big| x/|x|^{1+1/\rho}-2^ke_1  \big|\leq (1+\epsilon_0)2^{\rho k}$ for large $k$.
Thus, we have \begin{equation*}
\big\Vert S_kg_k\big\Vert_{L^{\infty}}\gtrsim \sigma_k2^{mk}2^{kd(1-\rho)/2} 
\end{equation*}
and
\begin{equation*}
\Vert g_k\Vert_{L^{\infty}}\approx \sigma_k.
\end{equation*}

When $m>-d(1-\rho)/2$ then put $\sigma_k=2^{-kd(m+d(1-\rho)/2)/2}$ to get
\begin{equation*}
\Big(\sum_{k=10}^{\infty}{\Vert g_k\Vert^q_{L^{\infty}}}\Big)^{1/q}\lesssim 1
\end{equation*} and 
\begin{equation*}
\Big( \sum_{k=10}^{\infty}{\big\Vert  S_kg_k \big\Vert^t_{L^{\infty}}}\Big)^{1/t} =\infty
\end{equation*} for all $0<q\leq\infty$ and $0<t\leq\infty.$ This proves Theorem \ref{sharptheorem2} (1).

When $m=-d(1-\rho)/2$ and $q>t$ then we put $\sigma_k=k^{-1/t}$, which proves Theorem \ref{sharptheorem2} (2).

\subsection{Proof of Theorem \ref{sharptheorem1}; The case $p=\infty$}
Now suppose $m>-d(1-\rho)/2$ and prove Theorem \ref{sharptheorem1} (1).
We pick $\epsilon>0$ such that $m-2\epsilon>-d(1-\rho)/2$ and proceed the argument in subsection \ref{ghghghgh}.
Due to Theorem \ref{sharptheorem2} (1) we can choose $g\in B_{\infty}^{\epsilon,\infty}$ so that
$\Vert g\Vert_{B_{\infty}^{\epsilon,\infty}}<\infty$ and $\big\Vert c_{m,\rho}(D)g \big\Vert_{B_{\infty}^{-\epsilon,\infty}}=\infty$.
Then we apply
 $B_{\infty}^{\epsilon,\infty}\hookrightarrow F_{\infty}^{0,q} $ and $F_{\infty}^{0,t} \hookrightarrow B_{\infty}^{-\epsilon,\infty}$.

\section{Additional remarks}\label{type11}
It is natural to ask about the boundedness of $T_a\in Op\mathcal{S}_{1,1}^{m}$ in $F_p^{s,q}$ and in this case it has many different situations. 
We cannot guarantee that the adjoint operator of $T_a\in Op\mathcal{S}_{1,1}^{m}$ belongs to the same type,
 and therefore the extension of $T_a\in Op\mathcal{S}_{1,1}^{m}$ to operators acting on $S'$ is not valid.   Moreover, the composition properties of pseudo-differential operators in Section \ref{property} does not work because the symbolic calculus cannot be applied to the case $\rho=\delta=1$ and thus we do not have the freedom of dependence of $s$.
Actually, Ching \cite{Chi} proves that not all operators of order $m=0$ are $L^2(=L^2_0)$ continuous and 
Stein proved that all operators in $Op\mathcal{S}_{1,1}^{0}$ are bounded on $H^{{s}}(=L_{{s}}^2)$ under the assumption $s>0$ in his unpublished work and Meyer \cite{Me} improved this result  by proving the continuity of $Op\mathcal{S}_{1,1}^{m}$-operators from ${L}^p_{{s}+m}$ to ${L}^p_{{s}}$ with the same assumption ${s}>0$ for $1<p<\infty$.        
For $0<p<\infty$ and $0<q\leq\infty$ Runst \cite{Ru}, Torres \cite{To}, Johnsen \cite{Joh} extended the continuity to Triebel-Lizorkin spaces $F_p^{s,q}$ with the condition $s>\max{\{0,d(1/p-1),d(1/q-1)\}}$.  Let ${\tau}_{p,q}=\max{\{0,d({1}/{p}-1),d({1}/{q}-1)\}}$ and ${\tau}_{p}=\max{\{0,d({1}/{p}-1)\}}$. 
\begin{theorem}\label{their}
Let $m\in\mathbb{R}$, $0<p<\infty$, and $0<q\leq \infty$. Suppose $a\in\mathcal{S}_{1,1}^{m}$.
\begin{enumerate}
\item  $T_{a}$ maps $F_{p}^{s+m,q}$ to $F_{p}^{s,q}$ if  $s>{\tau}_{p,q}$ ;
\item $T_{a}$ maps $B_{p}^{s+m,q}$ to $B_{p}^{s,q}$  if  $s>{\tau}_{p}$.
\end{enumerate}
\end{theorem}
(2) was proved in \cite{Joh, Ru} and (1) was done in \cite{Joh, Ru, To}.
Torres \cite{To} applied atoms and molecules for $F_{p}^{s,q}$ to prove the above result. Every $f\in F_{p}^{s,q}$ can be written as $f=\sum_{Q}{s_Q A_Q}$ where $\{s_Q\}_{Q}$ is a sequence of complex numbers in $f_{p}^{s+m,q}$ and $A_Q$'s are atoms for $F_{p}^{s+m,q}$ ( see \cite{Fr_Ja} for more details). Then he defined $\widetilde{T_{a}}f(x):=\sum_{Q}{s_Q T_{a}A_Q}$ and proved that $\widetilde{T_{a}}$ maps $F_{p}^{s+m,q}$ to $F_p^{s,q}$ by showing
$T_{a}$ maps atoms for $F_{p}^{s+m,q}$ to molecules for $F_{p}^{s,q}$. His argument also works for $p=\infty$.
Note that $\widetilde{T_{a}}$ agrees with $T_{a}$ on $F_p^{s,q}$ for $0<p<\infty$ by the density of $S$ in  $F_p^{s,q}$ for $0<p,q<\infty$. However, we should be careful to say that this implies the boundedness of the operator $T_{a}$ when $p=\infty$ because $T_{a}$ is not continuous on $S'$ as we mentioned above. In fact, a rigorous definition of $T_{a}$ on $S'$ was first given in \cite{Joh1} by using a limiting argument; Let $a_{j,k}(x,\xi)=\phi_j\ast a(\cdot ,\xi)(x)\widehat{\phi_k}(\xi)$ be defined as in (\ref{defaa}). Then we define for $f\in S'$
\begin{equation}\label{defdef}
T_{a}f=\lim_{N\to\infty}{\sum_{k=0}^{N}{\sum_{j=0}^{N}{T_{a_{j,k}}f}}}
\end{equation}  whenever  the limit converges in $S'$.

Recently the author \cite{Park4} proved that Theorem \ref{their} is sharp in the sense that if $s\leq \tau_{p,q}$ ($s\leq \tau_p$), then the boundedness results in $F_{p}^{s,q}$ ($B_p^{s,q}$) do not work by applying a random construction technique in \cite{Ch_Se}. The author \cite{Park1} also proved  a $``F_{\infty}^{s,q}$-$variant"$ of Fefferman-Stein maximal inequality, and  in  \cite{Park4}  extend Theorem \ref{their} to $p=\infty$ with the adaption of (\ref{defdef})  by using the maximal inequality.

\section{\textbf{Appendix - Compoud symbols}}\label{adjoint}

A useful device in the study of pseudo-differential operator is the notion of compound symbols, motivated by adjoint operators.
Observe that by writing out the Fourier transform we can rewrite (\ref{symbolcal}) as
\begin{equation}\label{rewrite}
T_af(x)=\int_{\mathbb{R}^d\times\mathbb{R}^d}{a(x,\xi)f(y)e^{2\pi i\langle x-y,\xi\rangle }  }dy d\xi.
\end{equation} Then the adjoint operator $(T_a)^*$ can be expressed as
\begin{equation*}
(T_a)^*f(x)=\int_{\mathbb{R}^d\times\mathbb{R}^d}{\overline{a}(y,\xi)f(y)e^{2\pi i\langle x-y,\xi\rangle }  }dy d\xi
\end{equation*} and this is not quite in the form (\ref{rewrite}) as the amplitude $\overline{a}(y,\xi)$ is not a function of $(x,\xi)$. Alternatively, we introduce compound symbols.
For $\rho,\delta_1,\delta_2, m\in\mathbb{R}$ a compound symbol $A$ in $\mathcal{S}_{\rho,\delta_1,\delta_2}^{m}$ is a smooth function defined on $\mathbb{R}^d\times\mathbb{R}^d\times\mathbb{R}^d$, satisfying the analogue of (\ref{symbolest}):  for all multi-indices $\alpha$, $\beta$, and $\gamma$ there exists a constant $c_{\alpha,\beta,\gamma}$ such that 
\begin{equation*}
\big| \partial_{\xi}^{\alpha}\partial_{x}^{\beta}\partial_{y}^{\gamma}A(x,y,\xi)  \big|\leq c_{\alpha,\beta,\gamma}\big(  1+|\xi|\big)^{m-\rho|\alpha|+\delta_1|\beta|+\delta_2|\gamma|} ~\text{for}~ (x,y,\xi)\in\mathbb{R}^d\times\mathbb{R}^d\times\mathbb{R}^d.
\end{equation*} The pseudo-differential operator $T_{[A]}$ corresponding to a compound symbol $A$ is defined by
\begin{equation*}
T_{[A]}f(x)=\int_{\mathbb{R}^d\times\mathbb{R}^d}{A(x,y,\xi)f(y)e^{2\pi i\langle x-y,\xi\rangle }}dyd\xi,\quad \text{for}~f\in S(\mathbb{R}^d).
\end{equation*}
Denote by $Op\mathcal{S}_{\rho,\delta_1,\delta_2}^{m}$ the class of pseudo-differential operators with compound symbols in $\mathcal{S}_{\rho,\delta_1,\delta_2}^{m}$.
Then the adjoint operator $(T_{[A]})^*$ is also a pseudo-differential operator \begin{equation*}
(T_{[A]})^*=T_{[A^*]}
\end{equation*} where $A^*(x,y,\xi)=\overline{A}(y,x,\xi)$.

\subsection{Composition of pseudo-differential operators}
Let $a\in\mathcal{S}_{\rho_0,\delta_0}^{m_0}$ and $A\in \mathcal{S}_{\rho_1,\delta_1,\delta_2}^{m_1}$.  Then we get 
\begin{equation*}
T_a\circ T_{[A]}f(x)=T_{[B]}f(x)
\end{equation*} where \begin{equation}\label{ccdd}
B(x,y,\xi)=\int_{\mathbb{R}^d\times\mathbb{R}^d}{ a(x,\eta)A(z,y,\xi)e^{2\pi i\langle x-z,\eta-\xi\rangle }  }d\eta dz.
\end{equation} 
Observe that if we put $A_y(x,\xi):=A(x,y,\xi)$ and $B_y(x,\xi):=B(x,y,\xi)$ for fixed $y$ then 
\begin{equation*}
T_{a}\circ T_{A_y}=T_{B_y}
\end{equation*} and $A_y\in\mathcal{S}_{\rho_1,\delta_1}^{m}$ uniformly in $y$. This implies that
$B_y$ belongs to $\mathcal{S}_{{\rho},{\delta} }^{m_0+m_1}$ uniformly in $y$  if
\begin{equation}\label{compcon}
0\leq \delta_1<\rho_0\leq 1
\end{equation} holds where
${\rho}=\min{(\rho_0,\rho_1)}  $ and ${\delta}=\max{(\delta_0,\delta_1)}$. That is,
\begin{equation*}
\big| \partial_{\xi}^{\alpha}\partial_x^{\beta}B(x,y,\xi)\big|\lesssim_{\alpha,\beta}(1+|\xi|)^{m_0+m_1-{\rho}|\alpha|+{\delta}|\beta|}
\end{equation*} uniformly in $y$.
Since $y$ involves only in $A$ on the right hand side of (\ref{ccdd}), we can replace $A$ by $\partial_{y}^{\gamma}A\in \mathcal{S}_{\rho_1,\delta_1,\delta_2}^{m_1+\delta_2|\gamma|}$ to get the estimate
\begin{equation*}
\big| \partial_{\xi}^{\alpha}\partial_x^{\beta}\partial_y^{\gamma}B(x,y,\xi)\big|\lesssim_{\alpha,\beta}(1+|\xi|)^{m_0+m_1-{\rho}|\alpha|+{\delta}|\beta|+\delta_2|\gamma|}.
\end{equation*}
We conclude that if (\ref{compcon}) holds then \begin{equation}\label{comcom}
T_a\circ T_{[A]}=T_{[B]}
\end{equation} where $B\in \mathcal{S}_{{\rho},{\delta},\delta_2}^{m_0+m_1}$.

\subsection{Analogue of Theorem \ref{main}, \ref{besov}, \ref{mmm}, \ref{mmmm}}

We will show  that for $0\leq \delta \leq \rho<1$ every pseudo-differential operator $T_{[A]}\in Op\mathcal{S}_{\rho,\delta,\delta}^{m}$ can be written as $T_{[A]}=T_a$ for some $a\in\mathcal{S}_{\rho,\delta}^{m}$ and consequently  we obtain the boundedness results of $Op\mathcal{S}_{\rho,\rho,\rho}^{m}$ for $0<\rho<1$.

\begin{lemma}\label{lllmmm}
Let $0\leq  \delta \leq \rho<1$ and $m\in\mathbb{R}$. 
Then every pseudo-differential operator $T_{[A]}$ corresponding to $A\in\mathcal{S}_{\rho,\delta,\delta}^{m}$ can be written as $T_{[A]}=T_a$ for some $a\in\mathcal{S}_{\rho,\delta}^{m}$.
\end{lemma}

We remark that when $A(x,y,\xi)=\overline{a}(y,\xi)$, Lemma \ref{lllmmm} proves that the adjoint operator of $T_a\in Op\mathcal{S}_{\rho,\delta}^{m}$ also belongs to the same type of class when $0\leq \delta\leq \rho<1$.

\begin{proof}[Proof of Lemma \ref{lllmmm}]
Let $A\in\mathcal{S}_{\rho,\delta,\delta}^{m}$.
Observe that for $f\in S(\mathbb{R}^d)$
\begin{align*}
T_{[A]}f(x)&= \int_{\mathbb{R}^d}{ \widehat{f}(\xi)e^{2\pi i\langle x,\xi\rangle }\Big(  \int_{\mathbb{R}^d\times\mathbb{R}^d}{A(x,y,\eta)  e^{-2\pi i<x-y,\xi-\eta>}  }d\eta dy  \Big)   }d\xi\\
     &= \int_{\mathbb{R}^d}{ \widehat{f}(\xi)e^{2\pi i<x,\xi>}\Big(  \int_{\mathbb{R}^d\times\mathbb{R}^d}{A(x,x-y,\xi-\eta)  e^{-2\pi i\langle y,\eta\rangle }  }d\eta dy  \Big)   }d\xi
\end{align*} and our claim is that
\begin{equation*}
a(x,\xi):=   \int_{\mathbb{R}^d\times\mathbb{R}^d}{A(x,x-y,\xi-\eta)  e^{-2\pi i\langle y,\eta\rangle }  }d\eta dy 
\end{equation*} belongs to $\mathcal{S}_{\rho,\delta}^{m}$. Here the integral is  an oscillatory integral and thus it can be interpreted as
\begin{equation*}
\lim_{\epsilon\to 0^+}  \int_{\mathbb{R}^d\times\mathbb{R}^d}{A(x,x-y,\xi-\eta)\mathcal{W}\big(\epsilon y   ,\epsilon \eta\big)  e^{-2\pi i\langle y,\eta\rangle }  }d\eta dy 
\end{equation*} where $\mathcal{W}\in\mathcal{C}_0^{\infty}(\mathbb{R}^d\times\mathbb{R}^d)$ is a test function which is identically $1$ on a neighborhood of the origin, and the limit is taken in the sense of tempered distributions.

We need to show that for any multi-indices $\alpha$ and $\beta$ there exists a constant $C_{\alpha,\beta}>0$ such that \begin{equation*}
\big| \partial_{\xi}^{\alpha}\partial_x^{\beta}a(x,\xi)  \big|\leq C_{\alpha,\beta}(1+|\xi|)^{m-\rho|\alpha|+\delta |\beta|}
\end{equation*} and we shall prove the case $\alpha=\beta=0$, otherwise apply the proof by replacing $m$ by $m-\rho |\alpha|+\delta|\beta|$. Furthermore, we may assume $m=0$ in the following reason.
As we did in Section \ref{property}, when $n_{-m}(\xi)=\big(  1+|\xi|^2 \big)^{-m/2}$ we have 
\begin{equation*}
n_{-m}(D)T_{[A]}\in Op\mathcal{S}_{\rho,\delta,\delta}^{0}.\end{equation*} by (\ref{comcom}). Once we prove 
$n_{-m}(D)T_{[A]}=T_a$ for some $a\in\mathcal{S}_{\rho,\delta}^{0}$, then 
\begin{equation*}
T_{[A]}=n_{m}(D)n_{-m}(D)T_{[A]}=n_{m}(D)T_a\in Op\mathcal{S}_{\rho,\delta}^{m}
\end{equation*} by (\ref{comcomcom}).
Thus assume $m=0$ and then it suffices to show
\begin{equation}\label{appgoal}
\big|a(x,\xi)\big|\lesssim 1.
\end{equation}
For $x\in \mathbb{R}^d$ and $R>0$ denote by $B(x,R)$ a ball of radius $R$ centered at $x$.

We first prove the case $\delta=\rho=0$. 
There exist a constant $C_d>0$ and a test function ${\chi}$ such that  $\chi$ is supported in $B(0,C_d)$ and  satisfies 
\begin{equation*}
\sum_{l\in \mathbb{Z}^d}{{\chi}(y-l)}=1
\end{equation*} for all $y \in\mathbb{R}^d$.
By using this we write
\begin{align}\label{fivedecomp}
a(x,\xi)&= \sum_{l,j\in\mathbb{Z}^d}{\int_{\mathbb{R}^d\times\mathbb{R}^d}{ A(x,x-y,\xi-\eta){\chi}(y-l){\chi}(\eta-j)e^{-2\pi i\langle y,\eta\rangle }  }d\eta dy}\nonumber\\
    &= \sum_{|l|,|j|\leq 10 C_d}{\dots}+\sum_{|l|\leq 10 C_d<|j|}{\dots}+\sum_{|j|\leq 10C_d<|l|}{\dots}+\sum_{10C_d<|l|\leq |j|}{\dots}+\sum_{10C_d<|j|<|l|}{\dots}\nonumber\\
    &:= I+II+III+IV+V.
\end{align}
Observe that 
for each $l,j\in\mathbb{Z}^d$ the integral does not exceed a constant times
\begin{equation}\label{integralbound}
\begin{cases}
1, \quad &
\quad \text{uniformly in}~l,j 
\\
 |j|^{-M}, \quad &
\quad  \text{for}~ |j|>10C_d 
\\
 |l|^{-M}, \quad &
\quad  \text{for}~ |l|>10C_d.
\end{cases}
\end{equation} for any $M>0$.
The first bound is immediate and the others follow from an integration by parts in each variable. Of course an implicit constant for the last two bounds depends on $M$. 
We use the bound $1$ for $I$, $|j|^{-M}$ for $II,IV$, and $|l|^{-M}$ for $III,V$ with $M>2d$ and then we establish
(\ref{appgoal}).

Now assume $0<\rho<1$. Define $A_{k}(x,y,\xi):=A(x,y,\xi)\widehat{\phi_k}(\xi)$ where $\phi_k$ is an auxiliary function we used to define $F$- and $B$-norms  previously. Then
\begin{equation*}
\widetilde{A_k}(x,y,\xi):=A_k(2^{-\rho k}x,2^{-\rho k}y, 2^{\rho k}\xi)
\end{equation*} belongs to $\mathcal{S}_{0,0,0}^{0}$ uniformly in $k$ because $\delta\leq\rho$.
Since $\{\phi_k\}$ forms a partition of unity, we write
\begin{align*}
a(x,\xi)&= \sum_{k=0}^{\infty}{\int_{\mathbb{R}^d\times\mathbb{R}^d}{ A_{k}(x,x-y,\xi-\eta)e^{-2\pi i\langle y,\eta\rangle } }d\eta dy}\\
    &= \sum_{k=0}^{\infty}{\int_{\mathbb{R}^d\times\mathbb{R}^d}{ \widetilde{A_{k}}\big(2^{\rho k}x,2^{\rho k}(x-y),2^{-\rho k}(\xi-\eta)\big)e^{-2\pi i\langle y,\eta\rangle } }d\eta dy}\\
     &= \sum_{k=0}^{\infty}{\int_{\mathbb{R}^d\times\mathbb{R}^d}{ \widetilde{A_{k}}\big(2^{\rho k}x, 2^{\rho k}x-y,2^{-\rho k}\xi-\eta\big)e^{-2\pi i\langle y,\eta\rangle } }d\eta dy}.
\end{align*}

First consider the case $|\xi|\geq 20C_d$ and choose an integer $k_0$ satisfying \begin{equation*}
2^{k_0}\leq |\xi|<2^{k_0+1}.
\end{equation*} 
Then break up $a(x,\xi)$ as
\begin{equation}\label{sums}
\sum_{0\leq k<\log_2{(10C_d)}}{\dots}+\sum_{\log_2{(10C_d)}\leq k\leq k_0}{\dots}+\sum_{k=k_0+1}^{\infty}{\dots}.\end{equation}
Since $\widetilde{A_k}$ is contained in $\mathcal{S}_{0,0,0}^{0}$ uniformly in $k$, each integral in the summations is bounded by a constant independent of $k$ and thus the first one is done because it is a finite sum.

Let $\Theta$ be a Schwartz function such that $0\leq \Theta \leq 1$, $\Theta(\xi)=1$ on $\{\xi: 2^{-1}\leq |\xi|\leq 2\}$ and $Supp(\Theta)\subset \{\xi:2^{-2}\leq |\xi|\leq 2^2\}$.
Then the second one of (\ref{sums}) is equal to 
\begin{align}\label{secondsums}
& \sum_{\log_2{(10C_d)}\leq k\leq k_0}{ \int_{\mathbb{R}^d\times\mathbb{R}^d}{     \widetilde{A_{k}}\big(2^{\rho k}x,2^{\rho k}x-y,2^{-\rho k}\xi-\eta\big)\Theta\big(2^{-k}(\xi-2^{\rho k}\eta)  \big)    }     }\nonumber\\
& \times e^{-2\pi i\langle y,\eta\rangle }  d\eta dy\nonumber\\
&= \sum_{\log_2{(10C_d)}\leq k\leq k_0}\sum_{l,j\in\mathbb{Z}^d}{       \int_{\mathbb{R}^d\times\mathbb{R}^d}{     \widetilde{A_{k}}\big(2^{\rho k}x, 2^{\rho k}x-y,2^{-\rho k}\xi-\eta\big)\Theta\big(2^{-k}(\xi-2^{\rho k}\eta)  \big)    }       }\nonumber\\
  & \times {\chi}(y-l) {\chi}(\eta-j) e^{-2\pi i\langle y,\eta\rangle }    d\eta dy \nonumber \\
&= \sum_{\log_2{(10C_d)}\leq k\leq k_0}\sum_{|j|\leq 2^{k_0(1-\rho)-2}}{\sum_{l}{\dots}}+\sum_{\log_2{(10C_d)}\leq k\leq k_0}\sum_{|j|>2^{k_0(1-\rho)-2}}{\sum_{l}{\dots}}
\end{align}
Observe that for $|j|\leq 2^{k_0(1-\rho)-2}$ \begin{equation*}
\big| \partial^{\alpha}\Theta\big( 2^{-k}(\xi-2^{\rho k}\eta)    \big)  \big|\lesssim_{\alpha,N}  2^{-N(k_0-k)} 
\end{equation*} for all $N>0$ and all multi-indices $\alpha$.
Then this observation and the method we used for the case $\rho=0$ yield that 
the first sum in (\ref{secondsums}) is less than a constant times
\begin{equation*}
\sum_{k=\log_2{(10C_d)}}^{k_0}{2^{-N(k_0-k)}}\lesssim 1
\end{equation*} for sufficiently large $N$.
We can also prove that the second sum in (\ref{secondsums}) is less than a constant by using (\ref{integralbound}) ( $|j|^{-M}$ when $|l|\leq |j|$, and $|l|^{-M}$ when $|l|>|j|$ ), which completes the estimate of second term in (\ref{sums}).

We write the last term of (\ref{sums}) as 
\begin{equation*}
\sum_{k=k_0+1}^{\infty}{ \int_{\mathbb{R}^d\times\mathbb{R}^d}{     \widetilde{A_{k}}\big(2^{\rho k}x,2^{\rho k}x-y,2^{-\rho k}\xi-\eta\big)\Theta\big(2^{-k}(\xi-2^{\rho k}\eta)  \big) e^{-2\pi i\langle y,\eta\rangle }    }d\eta dy     }.
\end{equation*}
Then as we did, decompose  the sum as
\begin{align*}
& \sum_{k=k_0+1}^{\infty}{\sum_{l,j\in\mathbb{Z}^d}{\dots}}\\
&= \sum_{k=k_0+1}^{\infty}{I}+\sum_{k=k_0+1}^{\infty}{II}+\sum_{k=k_0+1}^{\infty}{III}+\sum_{k=k_0+1}^{\infty}{IV}+\sum_{k=k_0+1}^{\infty}{V}
\end{align*} like (\ref{fivedecomp}).
We already know that each summand is bounded by a constant independent of $k$, but when we take a summation over $k_0+1\leq k$ it may diverge because it is an infinite sum.
In this case we observe  that 
\begin{equation}\label{keysum}
\sum_{k=k_0+1}^{\infty}{\Theta\big(2^{-k}(\xi-2^{\rho k}\eta)   \big)}
\end{equation} is a bounded function of $\eta$ because at most a finite number of supports of all summands have  an intersection and $\Theta$ itself is a bounded function. Furthermore, each derivative of (\ref{keysum}) is also a bounded function by the same reason. Since $\widetilde{A_k}$ belongs to $\mathcal{S}_{0,0,0}^{0}$ uniformly in $k$, we see that
\begin{align}\label{ababab}
& \sum_{k=k_0+1}^{\infty}{\big| \partial_{\eta}^{\alpha}\partial_y^{\beta}\widetilde{A_{k}}(2^{\rho k}x, 2^{\rho k}x-y,2^{-\rho k}\xi-\eta)\big|\big|\Theta\big( 2^{-k}(\xi-2^{\rho k}\eta)  \big)\big|}\nonumber\\
&\lesssim_{\alpha,\beta} \sum_{k=k_0+1}^{\infty}{\big|  \Theta\big( 2^{-k}(\xi-2^{\rho k}\eta)  \big)  \big|}\lesssim 1.
\end{align}
Finally (\ref{integralbound}) and (\ref{ababab}) prove that the last sum of (\ref{sums}) is less than a constant.

When $|\xi|\leq 20C_d$ then write $a(x,\xi)$ as 
\begin{equation*}
\sum_{0\leq k< \log_2{(10C_d)}}{\dots}+\sum_{\log_2{(10C_d)}\leq k}{\dots}.\end{equation*}
Then the first one is bounded by a constant because it is a finite sum and the estimate for the second one is derived from the fact that 
\begin{equation*}
\sum_{\log_2{(10C_d)}\geq k}{\Theta\big(2^{-k}(\xi-2^{\rho k}\eta)   \big)}
\end{equation*} is a bounded function of $\eta$.

\end{proof}

As a result of Lemma \ref{lllmmm} 
\begin{equation}
Op\mathcal{S}_{\rho,\delta}^{m}= Op\mathcal{S}_{\rho,\delta,\delta}^{m}
\end{equation} when $0\leq \delta\leq \rho<1$ and thus our boundedness results also holds for $Op\mathcal{S}_{\rho,\rho,\rho}^{m}$ when $0<\rho<1$.

\section*{Acknowledgement}

{
The author would like to thank his advisor Andreas Seeger for the guidance and helpful discussions.
The author also thanks the referee for many valuable remarks which greatly improved the presentation of this paper.
The author is supported in part by NSF grant DMS 1500162}.

\end{document}